\documentclass[10pt]{article}
\usepackage{latexsym}
\usepackage{amssymb,amsbsy,amsmath,amsfonts,amssymb,amscd}
\usepackage{mathrsfs,mathabx}
\usepackage{epsfig, graphicx}
\usepackage{graphics,epsfig}
\usepackage{color,url}
\usepackage{subfigure}

\usepackage{amsthm}
\usepackage{amsfonts}
\usepackage{graphicx}
\usepackage{subfigure}
\usepackage{amsmath}
\usepackage{amssymb}
\usepackage{amsmath}
\usepackage{bm}
\usepackage{cases}
\theoremstyle{plain}
\newtheorem{theorem}{Theorem}

\newtheorem{lemma}[theorem]{Lemma}

\theoremstyle{definition} 
\newtheorem{definition}[theorem]{Definition}
\newtheorem{remark}[theorem]{Remark}

\newcommand{\dvn}{\nabla_v}
\newcommand{\dvnstar}{\nabla_{v_*}}

\newcommand{\eps}{\varepsilon}
\newcommand{\half}{\frac{1}{2}}
\newcommand{\Qop}{\mathcal{Q}}
\newcommand{\vstar}{v_*}
\newcommand{\rd}{\mathrm{d}}

\newcommand{\RR}{\mathbb{R}}

\newcommand{\equilibrium}{\mathcal{M}}

\newcommand{\vecx}{v}
\newcommand{\vecy}{w}
\newcommand{\Aone}{A_{11}}
\newcommand{\aone}{a_{11}}
\newcommand{\aot}{a_{12}}
\newcommand{\aott}{a_{13}}
\newcommand{\attt}{a_{23}}
\newcommand{\atwo}{a_{22}}
\newcommand{\athree}{a_{33}}

\newcommand{\Dtwo}{D_{w_2}}
\newcommand{\Dthree}{D_{w_3}}

\textheight 9in \textwidth 6.5in \numberwithin{equation}{section}

\title{A particle method for the homogeneous Landau equation}
\author{Jose A. Carrillo\thanks{Mathematical Institute, University of Oxford, Oxford OX2 6GG, UK (carrillo@maths.ox.ac.uk).},
	Jingwei Hu\thanks{Department of Mathematics, Purdue University, West Lafayette, IN 47907, USA (jingweihu@purdue.edu).},
	Li Wang\thanks{School of Mathematics, University of Minnesota, Twin Cities, MN 55455, USA (wang8818@umn.edu).},
	Jeremy Wu\thanks{Department of Mathematics, Imperial College London, SW7 2AZ London, UK (jeremy.wu13@imperial.ac.uk).}
}

\begin{document}
	\maketitle
	
	\begin{abstract}
		We propose a novel deterministic particle method to numerically approximate the Landau equation for plasmas. Based on a new variational formulation in terms of gradient flows of the Landau equation, we regularize the collision operator to make sense of the particle solutions. These particle solutions solve a large coupled ODE system that retains all the important properties of the Landau operator, namely the conservation of mass, momentum and energy, and the decay of entropy. We illustrate our new method by showing its performance in several test cases including the physically relevant case of the Coulomb interaction. The comparison to the exact solution and the spectral method is strikingly good maintaining 2nd order accuracy. Moreover, an efficient implementation of the method via the treecode is explored. This gives a proof of concept for the practical use of our method when coupled with the classical PIC method for the Vlasov equation.
	\end{abstract}
	
	{\small 
		
	}

	\section{Introduction}
	
	The Landau equation is one of the fundamental kinetic equations describing the evolution of the distribution of charged particles in a collisional plasma \cite{Landau} where grazing collisions are predominant~\cite{DLD92,Vi98}. It is considered one of the most important equations in kinetic theory together with the Boltzmann equation, and it is of renewed interest in computational plasma physics due to the important applications related to fusion reactors and the ITER project.
	The Landau equation governs the evolution of the charged particles mass distribution function $f(t,x,v)$ in phase space $(x,v)\in \Omega\times\RR^d$ and is given by
	\begin{equation}  \label{fulllandau}
	\partial_t f  + v \cdot \nabla_x f +E\cdot \nabla_v f= \Qop(f,f) := \dvn \cdot \left\{ \int_{\RR^d} A(v-\vstar) \left( f(\vstar) \dvn f(v) - f(v) \dvnstar f(\vstar) \right) \rd \vstar\right\}\,,
	\end{equation}
	where $E$ is the acceleration due to external or self-consistent forces, the collision kernel takes the form $A(z) = |z|^{\gamma} \left(|z|^2 I_d - z \otimes z \right)= |z|^{\gamma+2} \Pi (z)$ with $I_d$ being the identity matrix, $\Pi(z)$ the projection matrix into $\{z\}^\perp$, $-d-1\leq \gamma \leq 1$, and $d\geq 2$. The most interesting case corresponds to $d=3$ with $\gamma = -3$ associated with the physical interaction in plasmas. This case is usually called the Coulomb case not because of the analogy of the singularity of the matrix $\|A(z)\|\simeq  |z|^{-1}$ at zero but because it can be derived from the Boltzmann equation in the grazing collision limit when particles interact via Coulomb forces \cite{DLD92}. The case $\gamma = 0$ is usually referred to as the Maxwellian case since the equation is reduced to a sort of degenerate linear Fokker-Planck equation preserving the same moments as the Landau equation \cite{ViMax}.
	
	This paper considers the numerical approximation of the Landau equation where the main focus is on the collision operator. Hence, for the rest of the paper, we shall consider exclusively the spatially homogeneous Landau equation
	\begin{equation} \label{landau}
	\partial_t f=\mathcal{Q}(f,f).
	\end{equation}
	The main formal properties of $\Qop$ rely on the following reformulation 
	\begin{equation} \label{Q}
	\Qop(f,f) = \dvn \cdot  \left\{ \int_{\RR^d} A(v-\vstar) f f_* \left (  \dvn \log f - \dvnstar \log f_* \right)\,  \rd \vstar\right\}\,,
	\end{equation}
	where $f=f(v)$, $f_*=f(\vstar)$ are used; and its weak form acting on appropriate test functions $\phi = \phi(v)$ 
	\begin{equation} \label{weak}
	\int_{\RR^d} \Qop(f,f) \phi \,\rd{v}= -\half \iint_{\RR^{2d}} (\dvn \phi - \dvnstar \phi_*)  \cdot  A(v-\vstar) \left(  \dvn \log f - \dvnstar \log f_*  \right) f f_* \, \rd v \,\rd\vstar\,.
	\end{equation}
	Then choosing $\phi(v) = 1, v, |v|^2$, one achieves conservation of mass, momentum and energy. Inserting $\phi(v) = \log f (v)$, one obtains the formal entropy decay with dissipation given by
	\begin{equation}
	\label{eq:entprod}
	\frac{d}{dt}\int_{\mathbb{R}^d}f \log f \,\rd{v} = -D(f(t,\cdot)):= -\frac{1}{2}\iint_{\RR^{2d}} B_{v,v_*} \cdot A(v-v_*)B_{v,v_*}ff_* \, \rd v\rd \vstar\leq 0\,,
	\end{equation}
	since $A$ is symmetric and semipositive definite, with $B_{v,v_*} := \dvn \log f - \dvnstar \log f_*$. The equilibrium distribution is given by the Maxwellian 
	\[
	\equilibrium_{\rho, u, T} = \frac{\rho}{(2\pi T)^{d/2}} \text{exp} \left(- \frac{ |v-u|^2}{2T} \right),
	\]
	for some constants $\rho, T$ determing the density and the temperature of the particle ensemble, and mean velocity vector $u$.
	We refer to \cite{Vi98,GZ} for its proof that we will recall in a regularized setting in Section 2.
	
	Besides the applications to physics, the Landau equation presents interesting mathematical challenges. The corresponding homogeneous equation arising from so-called hard potential cases ($\gamma \geq 0$) is by now very well understood in terms of existence, uniqueness, smoothing, decay, and moment/$L^p$ propagation owing primarily to the work of Desvillettes and Villani~\cite{MR1737547, MR1737548} and the references therein. One of the key ingredients is taking advantage of the finite entropy dissipation~\eqref{eq:entprod} which gives rise to the robust notion of `H-solution' as introduced by Villani~\cite{Vi98}. Less can be said at the moment about the soft potentials ($\gamma < 0$). The first major breakthrough in this direction was a global existence and uniqueness result by Guo~\cite{guo_landau_2002}. Of course, the result by Guo relied on many assumptions such as closeness to a Maxwellian for the initial data, high regularity, and small entropy. However, it remains difficult to weaken these assumptions while maintaining local existence or uniqueness. In the soft potential setting, there is even a dichotomy between the moderately soft potentials ($-2 \leq \gamma < 0$) and the extremely soft potentials ($\gamma < -2$). For example, Fournier and Gu\'erin were able to prove uniqueness of weak solutions using probabilistic techniques, yet additional initial moment and $L^p$ assumptions are needed as $\gamma$ becomes more negative with the result becoming only locally guaranteed when $\gamma$ is sufficiently negative~\cite{fournier_well-posedness_2009}. Incidentally, while their approach involved heavy probability machinery, they proved uniqueness through estimates involving the 2-Wasserstein distance, the fundamental quantity in the theory of gradient flows which is the perspective we adopt. An incomplete selection of contributions in the soft potential case that illustrate these difficulties is~\cite{Des1,Des2,ALL,CM,CDH,Wunotme,GZ,GZ1,GZ2,SW}. A cursory glance at some of these references highlights the variety of techniques needed to tackle the difficulties with soft potentials. Gualdani and her colleagues favour the degenerate parabolic perspective when viewing the Landau equation with radial symmetry~\cite{GG,GZ1}. The main issues of this approach are the quadratic non-linearity coming from the quadratic collision operator as well as the degeneracy of the diffusion matrix which depends on the solution. In~\cite{Des1,Des2}, Desvillettes obtains weighted Fisher information estimates depending on the dissipation~\eqref{eq:entprod}. More precisely, Desvillettes proved the estimate
	\[
	\int_{\RR^d}(1+|v|^2)^\frac{\gamma}{2}|\nabla\sqrt{f}|^2 \,\rd v \leq C(1+D(f(t,\cdot)))\,,
	\]
	where $C$ is a constant depending on the initial entropy, energy, and mass of $f$. For the soft potential case, $\gamma<0$, this estimate suggests the unavailability of an unweighted Fisher information bound. This hampers the standard methods passing through the Csisz\'ar-Kullback and logarithmic Sobolev inequalities to obtain rates of convergence to the Maxwellian equilibrium~\cite{T,TV}. Exponential convergence in the hard potential case is known~\cite{C} however, it seems that the analogous statement in the soft potential case currently holds only for the \textit{linearized} collision operator~\cite{GZ,CTW,CTWerr}. 
	
	We now turn to a new interpretation of the homogeneous Landau equation as a formal gradient flow on the set of probability measures. Following recent works in nonlinear Fokker-Planck equations \cite{CMV03,AGS,CLSS10} and the Boltzmann equation \cite{erbar2016}, we rewrite the homogeneous Landau equation as a nonlinear continuity equation where the velocity field is determined by the variational derivative of the entropy functional. More precisely, denoting by $E(f) = \int_{\RR^d} f \log f \,\rd v$ the entropy functional, we can rewrite \eqref{Q} as the nonlinear continuity equation
	\begin{equation*} 
	\Qop(f,f) = \dvn \cdot \left\{\left(\int_{\RR^d} A(v-\vstar) \left(  \dvn \frac{\delta E}{ \delta f} - \dvnstar \frac{\delta E_*}{\delta f_*} \right)\,  f_*  \rd \vstar \right) f \right\}\,,
	\end{equation*}
	and \eqref{weak} as
	\begin{equation*} 
	\int_{\RR^d} \Qop(f,f) \phi \,\rd{v} = -\half \iint_{\RR^{2d}} (\dvn \phi - \dvnstar \phi_*)  \cdot A(v-\vstar) \left(  \dvn \frac{\delta E}{ \delta f} - \dvnstar \frac{\delta E_*}{\delta f_*} \right) f f_* \, \rd v\,\rd \vstar\,.
	\end{equation*}
	Here $\frac{\delta E}{ \delta f}=\log f$ is the variational derivative, modulo constants, of the entropy functional in the set of nonnegative densities with a fixed mass. Therefore, one can write a formal gradient flow structure relative to a distance defined by a different action functional to the Boltzmann equation \cite{erbar2016}. The theoretical approach using this action functional will be pursued elsewhere. For our purposes, the crucial point is to realize that the homogeneous Landau equation can now be formally regularized without changing its main conservation and dissipative properties by regularizing the entropy functional. This strategy was recently used in the case of nonlinear Fokker-Planck equations with success \cite{carrillo2017blob} from both theoretical and numerical viewpoints. Analogously to \cite{carrillo2017blob}, consider a mollifier, Gaussian for simplicity, given by
	\begin{equation}\label{mollifier}
	\psi_\varepsilon(v) = \frac{1}{(2\pi \varepsilon)^\frac{d}{2}}\exp\left(-\frac{|v|^2}{2\varepsilon}\right)\,,
	\end{equation}
	for any  $\varepsilon>0$, and the associated regularized entropy as
	\begin{equation}\label{regent}
	E_\varepsilon(f) = \int_{\mathbb{R}^d} (f\ast\psi_\varepsilon)\log (f\ast \psi_\varepsilon)\, \rd{v} \,.
	\end{equation}
	The corresponding homogeneous Landau equation is given by
	\begin{equation}\label{Landaureg}
	\partial_t f = \Qop_\varepsilon (f,f) := - \dvn \cdot( U_\varepsilon(f) f)\,,
	\end{equation}
	with
	\begin{align} \label{QQe}
	U_\varepsilon(f) := - \int_{\RR^d} A(v-\vstar) \left(  \dvn \frac{\delta E_\varepsilon}{ \delta f} - \dvnstar \frac{\delta E_{\varepsilon,*}}{\delta f_*} \right)  f_*\,  \rd \vstar \,.
	\end{align}
	It is now easy to realize that the nonlinear nonlocal velocity field $U_\varepsilon(f)$ associated to the homogeneous regularized Landau equation makes sense even for $f$ being a convex combination of a finite number of Dirac Deltas. This allows us to introduce a particle method associated to the regularized kernel \eqref{QQe}. We will show that the associated particle method keeps the same conservation properties at the discrete level as the Landau equation \eqref{landau} while it dissipates the regularized entropy functional \eqref{regent}.
	
	Concerning deterministic particle methods for diffusive-type equations, there have been several strategies in the literature by introducing suitable regularizations of the flux of the continuity equation \cite{Russo2}. The case of the heat equation $\tfrac{\partial \rho}{\partial t} = \Delta \rho$ was considered in \cite{DegondMustieles,Russo} by interpreting the Laplacian as induced by a velocity field $u$, $\Delta \rho = -\nabla \cdot (u \rho)$, $u = -\nabla \rho/\rho$, and regularizing the numerator and denominator separately by convolution with a mollifier. Well-posedness of the resulting system of ordinary differential equations and a priori estimates relevant to the method were studied in \cite{LacombeMasGallic} and extended to nonlinear diffusions subsequently \cite{oelschlager1990large,LionsMasGallic, MGallic}. Variations of these methods allowing the weights to change in time were also analyzed in \cite{DegondMasGallic1, DegondMasGallic2}. The main disadvantage of these existing deterministic particle methods is that, with the exception of \cite{LionsMasGallic} for the porous medium equation $\tfrac{\partial \rho}{\partial t} = \Delta \rho^2$, they do not preserve the gradient flow structure \cite{LionsMasGallic}. For further background on deterministic particle methods, we refer to the review \cite{Chertock}, and for particle methods applied to transport equations, we refer to \cite{CR,CP,CB}. As mentioned earlier, we have followed here the strategy in \cite{carrillo2017blob} of regularizing the free energy functional instead in order to keep the gradient flow structure at the particle method level.
	
	To approximate the Landau operator, a popular method is to use the Fourier-Galerkin spectral method \cite{PRT00}. This method takes advantage of the convolutional property of the collision integral so that the resulting method can be implemented efficiently using fast Fourier transform (FFT). To be specific, the total complexity of one time evaluation of the collision operator requires $O(N_v^d\log N_v)$ complexity, where $N_v$ is the number of Fourier modes in each velocity dimension. We also refer to \cite{FP02,Gamba1,Gamba2,Gamba3} for additional properties of spectral methods and applications to inhomogeneous problems by time splitting methods. As we shall see, the proposed particle method would require $O(N^{2})$ complexity, where $N$ is the total number of particles. Hence in terms of efficiency, it may not be as fast as the spectral method. However, it is able to preserve all the physical properties of the equation: positivity, conservation of mass, momentum, and energy, and entropy decay. This is in contrast to the spectral method, wherein the truncated Fourier approximation destroys the structure of the solution (only mass is conserved, no positivity, no conservation of momentum and energy, no entropy decay). Furthermore, $O(N^2)$ is the direct cost of the particle method (a naive implementation). With the help of the fast summation technique such as the treecode, this cost can be reduced to $O(N\log N)$. We will explore this acceleration in the current paper while an in-depth study will be deferred to future work. 
	
	It is important to mention that the particle-in-cell (PIC) method \cite{BL,HE,TMST} is currently the dominant method to solve the Vlasov-type equation (equation \eqref{fulllandau} without the collision term) which is essentially a particle method. Hence, our proposed method is a natural candidate to be coupled with the PIC methodology to yield an efficient Lagrangian solver for the full Landau equation. The numerical exploration of these ideas in inhomogeneous problems is certainly a research topic of great interest, constituting a major future direction. For completeness, we finally mention that Eulerian methods based on mesh discretizations in velocity have also been proposed preserving the main properties of the Landau operator in \cite{DL2,Lem,BCDL} and the references therein. However, they are more difficult to incorporate within the PIC approach for spatially inhomogeneous problems.
	
	In the next section, we will analyze the properties of the regularized homogeneous Landau operator \eqref{QQe}. Section 3 is devoted to the introduction of the particle method and its properties. We end up in Section 4 with a thorough numerical study of its performance, comparison to exact solutions, computable convergence order and simulations in cases of interest for homogeneous problems. Appendix A gives a short summary of exact BKW solutions of the Landau equation as a reference. Appendix B recalls the basic aspects of the treecode strategy and its application to our particle method.

	\section{Regularized Landau equation: basic properties and kernel}
	
	In this section, we explore some theoretical properties associated to the homogeneous Landau equation with a regularized entropy functional.  The nonlinearity of $\Qop$ makes it difficult to directly regularize $f$ in a structure preserving way. Instead, the regularization is introduced at the level of the entropy functional which then modifies the homogeneous Landau equation. As mentioned in the introduction, we define, for any given $\varepsilon>0$, the regularized entropy as in (\ref{regent}) acting on $L^1_+(\RR^d)$ functions. The functions $\psi_\varepsilon(v)$ are mollifiers, fixed to be Gaussians with centre of mass at the origin and variance-covariance matrix equal to $\varepsilon I$ as in \eqref{mollifier} for simplicity. Notice that the regularized entropy is well-defined and its first variation with respect to constant mass densities $f$ gives
	\begin{equation}\label{eq:1stvarEeps}
	\frac{\delta E_\varepsilon}{\delta f} = \psi_\varepsilon*\log (f*\psi_\varepsilon)\,, \quad \nabla_v \frac{\delta E_\varepsilon}{\delta f}= (\nabla \psi_\eps) *\log (f*\psi_\eps)\,,
	\end{equation}
	after some computations, see \cite{carrillo2017blob} for details. Accordingly, this modifies the Landau equation in~\eqref{landau} to \eqref{Landaureg} with the nonlocal nonlinear velocity field given by \eqref{QQe}.
	
	The aim of this section is to show that equation~\eqref{Landaureg} preserves important structural properties as with the original homogeneous Landau equation. To fix ideas, we introduce a preliminary notion of a weak solution which we can refine after proving the standard conservation properties. For $p>0$ we will say that $g\in L_p^1(\mathbb{R}^d)$ to mean
	\[
	\int_{\mathbb{R}^d}(1+|v|^p)|g(v)|\rd{v} < \infty.
	\]
	Let us define
	\[
	\kappa(\gamma) = \left\{
	\begin{array}{cl}
	4+\gamma, 	&-2 \leq \gamma  \leq 0	\\
	6+\frac{\gamma}{2}, 	&-4 \leq \gamma <-2
	\end{array}
	\right. .
	\]
	
	\begin{definition}
		[Weak $\eps$-solution]\label{defn:wkLandau}
		We say that a nonnegative $f\in C([0,T];L_{\kappa(\gamma)}^1(\mathbb{R}^d))$ (denoted $f(t,v)$ whenever a time derivative is involved or just $f$) is a weak $\eps$-solution to equation~\eqref{Landaureg} if for every $\phi \in C_0^\infty((0,T)\times \mathbb{R}^d)$ we have
		\begin{equation}
		\label{eq:1wk}
		\int_0^T\int_{\mathbb{R}^d}\partial_t \phi f(t,v)\, \rd{v} \,\rd{t} - \frac{1}{2}\int_0^T \iint_{\mathbb{R}^{2d}}(\nabla_v \phi - \nabla_{v_*}\phi_*)\cdot A(v-v_*)\left(  \dvn \frac{\delta E_\eps}{ \delta f} - \dvnstar \frac{\delta E_{\eps,*}}{\delta f_*} \right) f f_* \, \rd {v}\, \rd\vstar\, \rd{t} = 0.
		\end{equation}
	\end{definition}
	
	Let us investigate the meaning of the weighted $L_{\kappa}^1$ requirement on $f$. We claim this is sufficient to make sense of the triple integral in~\eqref{eq:1wk}. Here, we are mainly concerned with the soft potentials given by $-4 \leq \gamma \leq 0$.
	In particular, since $\kappa \geq 2$ we have
	\[
	\sup_{t\in[0,T]}\int_{\mathbb{R}^d}|v|^2 f(t,v)\,\rd{v}<\infty,
	\]
	which ensures the following bound
	\begin{equation}
	\label{eq:logquad}
	\sup_{t\in[0,T]}|\log [f(t,\cdot)*\psi_\eps](v)| \leq C(\eps) (1 + |v|^2)\,,
	\end{equation}
	where $C=C(\eps)>0$ is a uniform constant depending only on $\eps>0$. Estimate~\eqref{eq:logquad} is obtained by computations similar to~\cite[Lemma 2.6]{CC92}. If more constants are introduced, we recycle $C$ to absorb them. Now let us investigate $B_{v,v_*}^\eps := \nabla_v \frac{\delta E_\eps}{\delta f} - \nabla_{v_*} \frac{\delta E_{\eps,v_*}}{\delta f_*}$. By~\eqref{eq:1stvarEeps}, this has the form
	\[
	B_{v,v_*}^\eps = C(\eps) \int_{\mathbb{R}^d} \left((v-v') \psi_\eps(v-v') - (v_*-v')\psi_\eps(v_*-v')\right) \log (f*\psi_\eps)(v')\,\rd v'\,.
	\]
	Applying estimate~\eqref{eq:logquad} gives
	\[
	|B_{v,v_*}^\eps| \leq C(\eps)\int_{\mathbb{R}^d} \left|
	(v-v') \psi_\eps(v-v') - (v_*-v')\psi_\eps(v_*-v')
	\right|(1+|v'|^2) \,	\rd v'\,.
	\]
	Consider first the easier moderately soft potential case $\gamma \geq -2$ and recall $A(v-v_*) = |v-v_*|^{\gamma+2}\Pi(v-v_*)$ where $\Pi(z)$ is the projection into $\{z\}^\perp$. For every test function, we have the bound
	\begin{align*}
	&\left|(\nabla_v \phi - \nabla_{v_*}\phi_*)\cdot A(v-v_*)\left(  \dvn \frac{\delta E_\eps}{ \delta f} - \dvnstar \frac{\delta E_{\eps,*}}{\delta f_*} \right) \right|   	\\
	&\qquad\leq C(\eps,\phi,d,\gamma) (|v|^{2+\gamma} + |v_*|^{2+\gamma})\int_{\mathbb{R}^d} (
	|(v-v') \psi_\eps(v-v')| + |(v_*-v')\psi_\eps(v_*-v')|
	)(1+|v'|^2) \,	\rd v' \,.
	\end{align*}
	By the change of variables $v' \mapsto v-v'$ and $v' \mapsto v_* - v'$, we have the following estimate
	\[
	\int_{\mathbb{R}^d} (
	|(v-v') \psi_\eps(v-v')| + |(v_*-v')\psi_\eps(v_*-v')|
	)(1+|v'|^2) 	\rd v' \leq C(\eps)(1+ |v|^2+|v_*|^2)\,.
	\]
	This can be used to estimate the triple integral of~\eqref{eq:1wk} by
	\begin{align*}
	&\int_0^T \iint_{\mathbb{R}^{2d}}\left|(\nabla_v \phi - \nabla_{v_*}\phi_*)\cdot A(v-v_*)\left(  \dvn \frac{\delta E_\eps}{ \delta f} - \dvnstar \frac{\delta E_{\eps,*}}{\delta f_*} \right) \right| f f_*\,  \rd v\,\rd{v_*} \,\rd{t}  	\\
	&\qquad\qquad\leq C(\eps,\phi,d,\gamma) \int_0^T \iint_{\mathbb{R}^{2d}} (|v|^{2+\gamma} + |v_*|^{2+\gamma})(1+|v|^2+|v_*|^2)ff_*\, \rd v \,\rd v_* \,\rd t\,.
	\end{align*}
	In this case, the $\kappa(\gamma) = 4+\gamma$ weight becomes clear to ensure absolute integrability.
	
	Let us now turn to the very soft potential case $-4 \leq \gamma < -2$. The same trick above will not work because the weight $|v-v_*|^{2+\gamma}$ is singular. Instead, we split $|v-v_*|^{2+\gamma} = |v-v_*|^{1+\frac{\gamma}{2}}|v-v_*|^{1+\frac{\gamma}{2}}$ so that we have
	\begin{align*}
	&\left|(\nabla_v \phi - \nabla_{v_*}\phi_*)\cdot A(v-v_*)\left(  \dvn \frac{\delta E_\eps}{ \delta f} - \dvnstar \frac{\delta E_{\eps,*}}{\delta f_*} \right) \right|	\\
	&\qquad\qquad\leq C(\eps,d)  \frac{|\nabla_v \phi - \nabla_{v_*}\phi_*|}{|v-v_*|^{-(1+\frac{\gamma}{2})}}\int_{\mathbb{R}^d} |v-v_*|^{1+\frac{\gamma}{2}}\left|
	(v-v') \psi_\eps(v-v') - (v_*-v')\psi_\eps(v_*-v')
	\right|(1+|v'|^2) \,	\rd v'\,.
	\end{align*}
	Splitting the weight allows us to see that $-(1+\frac{\gamma}{2}) \in (0,1]$ in the very soft potential case so that 
	$$
	\frac{|\nabla_v \phi - \nabla_{v_*}\phi_*|}{|v-v_*|^{-(1+\frac{\gamma}{2})}}
	$$
	can be estimated by the $C^{1,-(1+\frac{\gamma}{2})}$ norm of $\phi$. For the remaining $|v-v_*|^{1+\frac{\gamma}{2}}$ term within the integral over $v'$ we use the mean value theorem with $\left|
	(v-v') \psi_\eps(v-v') - (v_*-v')\psi_\eps(v_*-v')
	\right|$ to smother the singularity. Indeed, due to the form of $\psi_\eps$, we have that
	\[
	|v-v_*|^{1+\frac{\gamma}{2}}\left|
	(v-v') \psi_\eps(v-v') - (v_*-v')\psi_\eps(v_*-v')
	\right| \leq C(\eps)|v-v_*|^{2+\frac{\gamma}{2}}(1+|\xi - v'|^2)|\psi_\eps(\xi-v')|\,,
	\]
	where $\xi \in [v,v_*]$. Substitute this inequality back and use the change of variables $v'\mapsto \xi-v'$ to obtain
	\begin{align*}
	&\left|(\nabla_v \phi - \nabla_{v_*}\phi_*)\cdot A(v-v_*)\left(  \dvn \frac{\delta E_\eps}{ \delta f} - \dvnstar \frac{\delta E_{\eps,*}}{\delta f_*} \right) \right|  	\\
	&\qquad\qquad\leq C(\eps,d,\phi) |v-v_*|^{2+\frac{\gamma}{2}}\int_{\mathbb{R}^d}|v'|^2
	\psi_\eps(v')(1+|\xi-v'|^2) \,	\rd v' .
	\end{align*}
	The integral produces a term that has growth bounded by $(1+|\xi|^4)$ depending on $\epsilon$. Since $\xi \in [v,v_*]$, we can estimate $|\xi|^4 \leq C(|v|^4 + |v_*|^4)$. Inserting this back into the triple integral of~\eqref{eq:1wk} finally yields
	\begin{align*}
	&\quad \int_0^T \iint_{\mathbb{R}^{2d}}\left|(\nabla_v \phi - \nabla_{v_*}\phi_*)\cdot A(v-v_*)\left(  \dvn \frac{\delta E_\eps}{ \delta f} - \dvnstar \frac{\delta E_{\eps,*}}{\delta f_*} \right) \right|f f_*\,  \rd v\,\rd{v_*} \,\rd{t}  	\\
	&\qquad\qquad\qquad\qquad \leq C(\eps,d,\phi) \int_0^T \iint_{\mathbb{R}^{2d}} (1+|v|^{6+\frac{\gamma}{2}}+|v_*|^{6+\frac{\gamma}{2}}) ff_*\, \rd v\, \rd v_*\, \rd t\,.
	\end{align*}
	
	Equation~\eqref{eq:1wk} can be tested against more general functions $\phi$. As in~\cite[Remark 8.1.1]{AGS}, an equivalent expression of~\eqref{eq:1wk} is
	\begin{equation}
	\label{eq:diffact}
	\frac{d}{dt}\int_{\mathbb{R}^d} \phi f(t,v)\, \rd{v} = -\frac{1}{2}\iint_{\mathbb{R}^{2d}}(\nabla_v \phi - \nabla_{v_*}\phi_*)\cdot A(v-v_*)B_{v,v_*}^\eps f f_*\,  \rd v\,\rd{v_*}, \quad \forall \phi \in C_0^\infty (\mathbb{R}^d).
	\end{equation}
	Furthermore, \cite[Lemma 8.1.2]{AGS} allows us to refine the solution to be weakly continuous $t \in [0,T]\mapsto f(t,\cdot) \in L_\kappa^1(\mathbb{R}^d)$ so that whenever $\phi \in C_0^2((0,T)\times \mathbb{R}^d)$,
	\begin{align}
	\label{eq:wkcts}
	\begin{split}
	&\int_{t_1}^{t_2}\int_{\mathbb{R}^d}\partial_t \phi f(t,v)\, \rd{v} \,\rd{t} - \frac{1}{2}\int_{t_1}^{t_2} \iint_{\mathbb{R}^{2d}}(\nabla_v \phi - \nabla_{v_*}\phi_*)\cdot A(v-v_*)B_{v,v_*}^\eps f f_* \,\rd v\,\rd{v_*}\, \rd{t} \\ &\qquad= \int_{\mathbb{R}^d} \phi(t_2,v)f(t_2,v)\,\rd{v} - \int_{\mathbb{R}^d}\phi(t_1,v)f(t_1,v)\,\rd{v}\,.
	\end{split}
	\end{align}
	
	\begin{lemma}\label{lem:contest}
		Let $\phi$ be an admissible test function and $f$ be a weak $\eps$-solution to~\eqref{Landaureg}. Assume further that
		\[
		\nabla_v\phi(v) - \nabla_{v_*}\phi(v_*) \in \ker A(v-v_*),
		\]
		then
		\[
		\frac{d}{dt} \int_{\mathbb{R}^d} \phi(v) f(t,v)\, \rd v = 0
		\]
		holds, and therefore $\int_{\mathbb{R}^d} \phi(v) f(t,v) \rd v$ is a conserved quantity.
	\end{lemma}
	\begin{proof}
		We begin with the formal computations. Differentiating in time, we get
		\[
		\frac{d}{dt}\int_{\mathbb{R}^d}\phi(v) f(t,v) \,\rd v = -\frac{1}{2}\iint_{\mathbb{R}^{2d}} (\nabla_v\phi(v) - \nabla_{v_*}\phi(v_*)) \cdot A(v-v_*)B_{v,v_*}^\varepsilon  ff_* \,\rd v\,\rd v_* 	= 0\,.
		\]
		
		To justify these formal computations, we appeal to smooth cut-off arguments to approximate $1,v,|v|^2$ by admissible test functions using~\eqref{eq:wkcts} to compare $\int_{\mathbb{R}^d}\phi(v)f(0,v)\rd{v}$ with $\int_{\mathbb{R}^d}\phi(v)f(t,v)\rd{v}$.
	\end{proof}
	
	Since the kernel of the matrix $A(z)$ is spanned by $z$, a direct consequence of the previous result is that the mass, momentum, and energy for weak $\eps$-solutions of~\eqref{Landaureg} are conserved, i.e., 
	\begin{equation}\label{conservations}
	\frac{d}{dt}\left(
	\int_{\mathbb{R}^d}f(t,v)\,\rd v, \int_{\mathbb{R}^d} vf(t,v)\,\rd v, \int_{\mathbb{R}^d} |v|^2f(t,v)\,\rd v\right)=0\,.
	\end{equation}
	
	In this way, we define the mass, momentum, and energy of $f$ for all times by the constants $\rho, u, T$ related in the following way
	\begin{equation}
	\label{eq:moments}
	\rho = \int_{\mathbb{R}^d} f(t,v)\,\rd v\,, \quad \rho u = \int_{\mathbb{R}^d}vf(t,v)\, \rd v\,, \quad \rho u^2 + \rho dT = \int_{\mathbb{R}^d} |v|^2f(t,v) \,\rd v\,.
	\end{equation}
	
	As promised, we can refine the notion of weak $\eps$-solution. We add a finite dissipation property which is a mild assumption but yields theoretical and numerical advantages in the spirit of Villani's H-solution~\cite{Vi98}. One example of the analytic benefits is in~\cite[Proposition 4.2]{erbar2016} where Erbar recovers a strong upper gradient notion for the Boltzmann equation.
		\begin{definition}
			[Dissipative $\eps$-solution]\label{defn:refwkLandau}We say that $f\in C([0,T];L_{\kappa}^1(\mathbb{R}^d))$ is a dissipative $\eps$-solution with moments $(\rho,u,T)\in \mathbb{R}_+\times \mathbb{R}^d\times \mathbb{R}_+$ under the relation~\eqref{eq:moments} to the regularized Landau equation~\eqref{Landaureg} if it is a weak $\eps$-solution in the sense of Definition~\ref{defn:wkLandau} and
			\begin{enumerate}
				\item For every $\phi \in C_0^\infty(\mathbb{R}^d)$, equation~\eqref{eq:diffact} holds:
				\[
				\frac{d}{dt}\int_{\mathbb{R}^d} \phi f(t,v)\, \rd{v} = -\frac{1}{2}\iint_{\mathbb{R}^{2d}}(\nabla_v \phi - \nabla_{v_*}\phi_*)\cdot A(v-v_*)B_{v,v_*}^\eps f f_*  \,\rd v\,\rd{v_*}\,.
				\] \label{item:1refwk}
				\item The initial entropy is finite
				\[
				E_\eps(f(0,\cdot)) = \int_{\mathbb{R}^d} (f_0*\psi_\eps) \log (f_0*\psi_\eps) \rd v < \infty.
				\] \label{item:2rewk}
				\item The entropy-dissipation associated to the regularized equation is integrable in time
				\begin{equation}
				\label{eq:finitediss}
				D_\eps(f(t,\cdot)) := \frac{1}{2}\iint_{\mathbb{R}^{2d}} B_{v,v_*}^\eps \cdot A(v-v_*) B_{v,v_*}^\eps ff_* \,\rd{v}\,\rd{v_*} \in L^1(0,T).
				\end{equation}\label{item:3rewk}
			\end{enumerate}
	\end{definition}
	With the notion of a dissipative $\eps$-solution in hand, our next result displays the natural consequences of items~\ref{item:2rewk} and~\ref{item:3rewk}.
	
	\begin{lemma}
		\label{lem:Enonincrease}
		Let $f$ be a dissipative $\eps$-solution of the regularized Landau equation \eqref{Landaureg} with the collision operator given by \eqref{QQe}, then we have:
			\begin{equation}\label{entdis}
			E_\eps(f(t,\cdot)) - E_\eps(f(0,\cdot)) = -\frac{1}{2}\int_0^t \iint_{\mathbb{R}^{2d}}ff_* B_{v,v_*}^\eps \cdot A(v-v_*) B_{v,v_*}^\eps \rd v \rd v_* \rd s\leq 0,
			\end{equation}
			for all times $t\geq 0$.
	\end{lemma}
	\begin{proof}
			To begin, let us pretend that $f$ satisfies~\eqref{Landaureg} pointwise with sufficient smoothness in time. Formally, we differentiate $E_\eps(f(t,\cdot))$ in time to obtain
			\begin{align*}
			\frac{dE_\eps}{dt} &= \frac{d}{dt}\left(
			\int_{\mathbb{R}^d} f*\psi_\eps \log (f*\psi_\eps) \rd v
			\right) = \int_{\mathbb{R}^d} (\partial_t f *\psi_\eps) (\log (f*\psi_\eps) + 1) \rd v  \\
			&= \int_{\mathbb{R}^d} \partial_t f (\psi_\eps * \log (f*\psi_\eps) + 1) \rd v   
			= -\int_{\mathbb{R}^d} \nabla_v\cdot(U_\eps(f)f) (\psi_\eps * \log (f*\psi_\eps) + 1) \rd v\\
			&=  \int_{\mathbb{R}^d} \nabla_v\cdot\left(
			f \int_{\mathbb{R}^d} f_* A(v-v_*) B_{v,v_*}^\eps \rd v_*
			\right)\left(\frac{\delta E_\eps}{\delta f}+1\right) \rd v     \\
			&= -\iint_{\mathbb{R}^{2d}} ff_* \nabla_v \frac{\delta E_\eps}{\delta f} \cdot A(v-v_*) B_{v,v_*}^\eps \rd v \rd v_* = -\frac{1}{2} \iint_{\mathbb{R}^{2d}} ff_* B_{v,v_*}^\eps \cdot A(v-v_*) B_{v,v_*}^\eps \rd v \rd v_*.
			\end{align*}
			In the last line we have used integration by parts as well as symmetrizing $v\leftrightarrow v_*$ to recover $B_{v,v_*}^\eps$ from $\nabla \frac{\delta E_\eps}{\delta f}$. Integrating the ends of this equality in time gives~\eqref{entdis}. To obtain the full result with the minimal time regularity of dissipative $\eps$-solutions, we appeal to a standard mollification in time argument.
	\end{proof}
	
	In the rest of this section, we follow the strategy of~\cite[Theorem 4]{GZ} and \cite[Lemma 3]{Vi96} to deduce that stationary states of the homogeneous regularized Landau equation \eqref{Landaureg} can be characterized by Maxwellians. Since we are working with weak $\eps$-solutions, let us be specific and define what we mean by stationary states.
	\begin{definition}
		[Stationary states]\label{defn:station}We say that a dissipative $\eps$-solution $f$ is a stationary state to the homogeneous regularized Landau equation~\eqref{Landaureg} if for every test function $\phi \in C_0^\infty(\mathbb{R}^d)$,
		\[
		\iint_{\mathbb{R}^{2d}}(\nabla_v \phi - \nabla_{v_*}\phi_*)\cdot A(v-v_*)B_{v,v_*}^\eps f f_*  \,\rd v\,\rd{v_*} = 0\,, \quad \forall t\in[0,T].
		\]
	\end{definition} 
	We can use this definition with Lemma~\ref{lem:Enonincrease} to characterize the first variation of the entropy for a stationary state.
	
	\begin{lemma}\label{lem:1stvarquad}
		If $f$ is a stationary state of the regularized Landau equation \eqref{Landaureg} or equivalently $f$ is in the kernel of \eqref{QQe}, then the first variation of $E_\varepsilon$ is a quadratic polynomial in $v$, that is
		\begin{equation}\label{eq:1stvarquad}
		\frac{\delta E_\varepsilon}{\delta f} = \lambda^{(0)} + \lambda^{(1)}\cdot v + \frac{\lambda^{(2)}}{2}|v|^2\,.
		\end{equation}
		The constants $\lambda^{(0)},\lambda^{(2)}\in \mathbb{R}$ and $\lambda^{(1)}\in\mathbb{R}^d$ (depending on $\varepsilon$) can be determined by the conserved quantities \eqref{conservations} (see later in Lemma~\ref{lem:ker=Gauss}).
	\end{lemma}
	\begin{proof}
		This proof adopts the strategy of~\cite[Theorem 4]{GZ}. Lemma~\ref{lem:Enonincrease} implies that the entropy-dissipation, the right-hand side of \eqref{entdis}, is zero. Moreover, the entropy-dissipation is zero if and only if the quadratic form in the integrand of the right-hand side of \eqref{entdis} is zero. By definition of $A(v-v_*)$, we must have that $B_{v,v_*}^\varepsilon$ belongs to the kernel of $A(v-v_*)$ which is characterized by those vectors which are linearly dependent with $v-v_*$. Thus, there exists $\lambda^{(2)}:\mathbb{R}^d \times \mathbb{R}^d \to \mathbb{R}$ with the property
		\begin{equation}\label{eq:11}
		\nabla_v\frac{\delta E_\varepsilon}{\delta f}(v) - \nabla_{v_*}\frac{\delta E_\varepsilon}{\delta f}(v_*) = \lambda^{(2)}(v,v_*)(v-v_*)\,.
		\end{equation}
		At this point, we study $\lambda^{(2)}$ and seek to show that the diagonal mapping is constant. Immediately from (\ref{eq:11}), we notice that $\lambda^{(2)}(v,v_*)=\lambda^{(2)}(v_*,v)$. For any $i,j \in \{1,\dots,d\}$ when looking at the $j^{th}$ coordinate of (\ref{eq:11}) and then differentiating with respect to $v_i$ (valid as the $\eps$ regularization grants arbitrary smoothness), we have
		\[
		\partial_{v_i}\partial_{v_j} \frac{\delta E_\varepsilon}{\delta f} = \partial_{v_i}\lambda^{(2)}(v,v_*)(v_j-v_{*j}) + \lambda^{(2)}(v,v_*)\delta_{ij}\,.
		\]
		Set $v=v_*$ in the above equation to deduce
		\begin{equation}\label{eq:12}
		\partial_{v_i}\partial_{v_j} \frac{\delta E_\varepsilon}{\delta f} = \lambda^{(2)}(v,v)\delta_{ij}\,.
		\end{equation}
		Differentiating (\ref{eq:12}) again with respect to $v_k$ for $k \in \{1,\dots,d\}$ yields
		\[
		\partial_{v_k}\partial_{v_i}\partial_{v_j} \frac{\delta E_\varepsilon}{\delta f} = \partial_{v_k}\lambda^{(2)}(v,v)\delta_{ij}\,.
		\]
		The partial derivatives on the left hand side of the above may be freely permuted with no change to the expression. More interesting is the permutation of the associated indices $i,j,$ and $k$ on the right hand side. One instance of this is the following identity
		$
		\partial_{v_k}\lambda^{(2)}(v,v)\delta_{ij} = \partial_{v_i}\lambda^{(2)}(v,v)\delta_{kj}.
		$ 
		For arbitrary indices $k\in\{1,\dots,d\}$, simply take $i=j\in\{1,\dots,d\}\setminus\{k\}$ and one sees from before that
		$
		\partial_{v_k}\lambda^{(2)}(v,v)=0.
		$
		Since $k\in\{1,\dots,d\}$ was arbitrary, this implies that $\lambda^{(2)}(v,v)$ is actually a constant which we shall refer to as $\lambda^{(2)}$. Equipped with this information, integrating (\ref{eq:12}) twice confirms the claim of the lemma that the first variation of the entropy is a quadratic polynomial given by \eqref{eq:1stvarquad} for some constants $\lambda^{(1)}\in \mathbb{R}^d$ and $\lambda^{(0)}\in \mathbb{R}$.
	\end{proof}
	
	Our next step is to show that if $f$ satisfies equation~\eqref{eq:1stvarquad} then it is a Maxwellian with explicitly computable mass, momentum, and energy.
	
	\begin{lemma}\label{lem:ker=Gauss}
		If $f\in L^1_+(\RR^d)\setminus\{0\}$ satisfies the following equation
		\[
		\frac{\delta E_\varepsilon}{\delta f} = \lambda^{(0)} + \lambda^{(1)}\cdot v + \frac{\lambda^{(2)}}{2}|v|^2\,,
		\]
		then it must be a Maxwellian, $f(v) = \mathcal{M}_{\rho,u,T}(v)$. We can deduce a restriction on $\lambda^{(2)}$, specifically, that $\eps|\lambda^{(2)}|<1$. Furthermore, the mass, momentum, and energy explicitly depend on $\varepsilon, \lambda^{(0)},\lambda^{(1)}$, and $\lambda^{(2)}$ in the following way:
		\begin{equation}
		\label{eq:fconstants}\left\{
		\begin{array}{ll}
		\rho 	&= \left(\frac{2\pi}{|\lambda^{(2)}|}\right)^\frac{d}{2}\exp\left\{
		\lambda^{(0)} + \frac{\varepsilon |\lambda^{(2)}|d}{2} - \frac{\varepsilon |\lambda^{(1)}|^2}{2(1-\varepsilon |\lambda^{(2)}|)} + \frac{|\lambda^{(1)}|^2}{2|\lambda^{(2)}|(1-\varepsilon |\lambda^{(2)}|)}
		\right\}	\\
		u 	&= \frac{\lambda^{(1)}}{|\lambda^{(2)|}} \\
		T 	&= \frac{1}{|\lambda^{(2)}|} - \varepsilon 
		\end{array}
		\right.\,.
		\end{equation}
	\end{lemma}
	\begin{proof}
		We iteratively Fourier transform equation~(\ref{eq:1stvarquad}) recalling in particular the convolution and inversion theorems (especially that Maxwellians are Fourier transformed to Maxwellians) to deduce the identities
		\begin{align*}
		\psi_\varepsilon*\log(f*\psi_\varepsilon)  	&= \lambda^{(0)} + \lambda^{(1)}\cdot v + \frac{\lambda^{(2)}}{2}|v|^2 	\,,\\
		\mathcal{F}\{\log(f*\psi_\varepsilon)\} 	&= (2\pi\varepsilon)^\frac{d}{2}\frac{1}{\psi_\frac{1}{\varepsilon}}\left(\lambda^{(0)}\delta_0 + i \lambda^{(1)}\cdot \nabla \delta_0 - \frac{\lambda^{(2)}}{2} \Delta\delta_0\right) \,, 	
		\end{align*}
		and 
		\begin{equation}
		\label{eq:auxpoly}
		\log(f*\psi_\varepsilon) = \lambda^{(0)} - \frac{\lambda^{(2)}\varepsilon d}{2} + \lambda^{(1)}\cdot v + \frac{\lambda^{(2)}}{2}|v|^2  	\,.
		\end{equation}
		At this point, we remark that the sign of $\lambda^{(2)}$ can be deduced. Specifically, we claim that $\lambda^{(2)} < 0$. The significance of this is to ensure that the exponential of the right-hand side of~\eqref{eq:auxpoly} is integrable, and therefore Fourier transformable. Firstly, $\lambda^{(2)} \leq 0$ because the Dominated Convergence Theorem yields $\lim_{|v|\to \infty} f*\psi_\varepsilon(v) = 0$. This means that the right-hand side of~\eqref{eq:auxpoly} must decrease to $-\infty$ in the limit $|v|\to \infty$. Therefore, looking at the leading order contribution of the right-hand side of~\eqref{eq:auxpoly} gives $\lambda^{(2)}\leq 0$. Suppose for a contradiction that $\lambda^{(2)} = 0$, so the leading order contribution sending the right-hand side of~\eqref{eq:auxpoly} to $-\infty$ is $\lambda^{(1)}\cdot v$. In other words, $\lambda^{(1)}\cdot v$ must converge to $-\infty$ whenever $|v|\to \infty$. However, we can always choose a sequence $v_n =n \frac{\lambda^{(1)}}{|\lambda^{(1)}|}$ for $n\in\mathbb{N}$ which is colinear to $\lambda^{(1)}$ satisfying $\lambda^{(1)}\cdot v_n \to \infty$ and $|v_n| \to \infty$ as $n\to \infty$.
		
		Taking exponentials of both sides of~\eqref{eq:auxpoly}, we have
		$$
		f*\psi_\varepsilon(v) 	= \exp\left\{\lambda^{(0)} - \frac{|\lambda^{(1)}|^2}{2\lambda^{(2)}} - \frac{\lambda^{(2)}\varepsilon d}{2}\right\}\exp\left\{\frac{\lambda^{(2)}}{2}\left|v+\frac{\lambda^{(1)}}{\lambda^{(2)}}\right|^2\right\} \,,
		$$
		and one more Fourier transform (valid by the sign of $\lambda^{(2)}< 0$ discussed in the previous paragraph) leads to
		$$
		\mathcal{F}\{f\}(\xi) 	= (2\pi)^d(1-\varepsilon|\lambda^{(2)}|)^{-\frac{d}{2}} \exp\left\{\lambda^{(0)} + \frac{|\lambda^{(2)}|\varepsilon d}{2} - \frac{\varepsilon |\lambda^{(1)}|^2}{2(1-\varepsilon|\lambda^{(2)}|)}\right\}\mathcal{M}_{\left(1,-\frac{i\lambda^{(1)}}{1-\varepsilon|\lambda^{(2)}|},\frac{|\lambda^{(2)}|}{1-\varepsilon|\lambda^{(2)}|}\right)}(\xi)\,,
		$$
		after tedious algebra to collect terms. Here, we are using the convention that, for vectors $x,y\in\mathbb{R}^d$, $|x+iy|^2 := |x|^2 + 2ix\cdot y - |y|^2$. By the Riemann-Lebesgue lemma, we know that $|\mathcal{F}\{f\}(\xi)|\to 0$ as $|\xi|\to\infty$. With the expression for $\mathcal{F}\{f\}$ above, this means that the variance of $\mathcal{M}$ (the third parameter in the subscript) must be strictly positive. Hence, $1- \eps |\lambda^{(2)}|>0$. One final Fourier inversion gives an expression for $f$ as
		\begin{equation*}
		f(v) = \left( \frac{2\pi}{|\lambda^{(2)}|}\right)^\frac{d}{2}\exp\left\{
		\lambda^{(0)} + \frac{\varepsilon|\lambda^{(2)}|d}{2} - \frac{\varepsilon|\lambda^{(1)}|^2}{2(1-\varepsilon|\lambda^{(2)}|)}+ \frac{|\lambda^{(1)}|^2}{2|\lambda^{(2)}|(1-\varepsilon|\lambda^{(2)}|)}\right\}\mathcal{M}_{1,\frac{\lambda^{(1)}}{|\lambda^{(2)}|},\frac{1}{|\lambda^{(2)}|} - \varepsilon
		}(v)\,.
		\end{equation*}
		Reading off the constants, one confirms~\eqref{eq:fconstants}. Note that in the determination of $\rho,u,T$ in equation~\eqref{eq:fconstants}, we have a one-to-one correspondence between $(\rho,u,T)$ and $(\lambda^{(0)},\lambda^{(1)},\lambda^{(2)})$. Indeed, $\lambda^{(2)}$ is determined from $T$ which then gives $\lambda^{(1)}$ in the equation for $u$. Finally, $\lambda^{(0)}$ is determined from the equation for $\rho$.
	\end{proof}
	The previous lemmas give the following equivalence.
	\begin{theorem}
		$f$ is a stationary state of the regularized Landau equation \eqref{Landaureg} if and only if $f$ is a Maxwellian with parameters given by~\eqref{eq:fconstants} depending on the quadratic polynomial in equation~\eqref{eq:1stvarquad}.
	\end{theorem}
	\begin{proof}
		$(\implies)$ This direction combines Lemmas~\ref{lem:1stvarquad} and~\ref{lem:ker=Gauss}.
		
		\noindent $(\impliedby)$ This direction is a computation of $\frac{\delta E_\eps}{\delta f} = \psi_\eps * \log (f*\psi_\eps)$ when $f$ is a Maxwellian.
	\end{proof}
	
	\begin{remark}
		An alternative regularization for the entropy is
		\begin{equation}\label{alter}
		\tilde{E}_\varepsilon(f) = \int_{\mathbb{R}^d} f\log (f*\psi_\varepsilon)\,\rd{v}\,.
		\end{equation}
		Lemma~\ref{lem:1stvarquad} is still true with this alternative regularized entropy where the first variation of $\tilde{E}_\eps$ and its gradient are given by, see \cite{carrillo2017blob},
		\begin{equation}\label{altervar}
		\frac{\delta \tilde{E}_\eps}{\delta f } = \log(f \ast \psi_\eps) + \left( \frac{f}{f \ast \psi_\eps} \right) \ast \psi_\eps\,, \quad
		\nabla_v \frac{\delta \tilde{E}_\eps}{\delta f } = \frac{ f \ast \dvn \psi_\eps}{f \ast \psi_\eps } + \left( \frac{f}{ f \ast \psi_\eps } \right)\ast \dvn \psi_\eps \,.
		\end{equation}
		However, the characterization result of Lemma~\ref{lem:ker=Gauss} as a Maxwellian is not true, even if one might expect the existence and uniqueness of a stationary state being the conserved quantities fixed. 
	\end{remark}
	
	
	\section{A particle method for the homogeneous Landau equation}
	
	The main idea is analogous to the recent work \cite{carrillo2017blob} for aggregation-diffusion equations. In fact, the regularized Landau equation \eqref{Landaureg} can be viewed as a convection in $v$ with velocity field given by \eqref{QQe}, and thus giving access to a particle formulation. More specifically,  denote 
	\begin{equation} \label{fN}
	f^N(t,v)= \sum_{i = 1}^Nw_i \delta (v - v_i(t))\,,
	\end{equation} 
	with $N$ being the total number of particles, $v_i(t)$ the velocity of particle $i$, and $w_i$ the weight of particle $i$. Plugging \eqref{fN} as a distributional solution to \eqref{Landaureg}, we obtain that the evolution for the particle velocities $v_i(t)$, $1\leq i \leq N$ is given by
	\begin{align} \label{vt2}
	\frac{d v_i(t)}{dt} & =  U_\eps (f^N) (t, v_i(t)) =  -  \sum_j w_jA(v_i - v_j) \left[  \nabla \frac{\delta E^N_\eps}{\delta f}(v_i) - \nabla \frac{\delta E^N_\eps}{\delta f}(v_j)  \right] \nonumber
	\\ & = -  \sum_j w_j A(v_i - v_j) \left\{   \int_{\RR^d}  \nabla \psi_\eps (v_i - v) \log \left( \sum_k w_k\psi_{\varepsilon}( v-v_k) \right) \,\rd  v \right. \nonumber
	\\ & \hspace{3.5cm}  - \left.   \int_{\RR^d}  \nabla \psi_\eps (v_j -  v) \log \left( \sum_k w_k\psi_{\varepsilon}( v-v_k) \right) \,\rd v \right\}\,,
	\end{align}
	with $
	\frac{\delta E^N_\eps}{\delta f}:= \psi_\varepsilon*\log (f^N*\psi_\varepsilon)
	$
	and therefore, 
	\begin{equation} \label{nabladeltaE}
	\nabla \frac{\delta E^N_\eps}{\delta f}(v_i) = \int_{\RR^d} \nabla \psi_\eps(v_i - v) \log \left(\sum_k w_k \psi_\eps (v-v_k) \right) \rd v\,.
	\end{equation}
	Let us show next that the semidiscrete particle method defined by \eqref{vt2} leads to a numerical particle approximation $f^N$ of the solution to the regularized Landau equation \eqref{Landaureg} conserving mass, momentum and energy and enjoying the regularized entropy dissipation \eqref{entdis}.
	
	\begin{theorem} \label{thm10}
		The semidiscrete particle method \eqref{vt2} satisfies the following properties:
		\begin{itemize}
			\item[1)] Conservation of mass, momentum, and energy: $\frac{d}{dt} \sum_{i=1}^N w_i \phi(v_i) = 0$ for $\phi(v_i) = 1, v_i, |v_i|^2$. 
			\item[2)] Dissipation of entropy: let 
			\begin{equation} \label{EN23}
			E^N_\eps = E_\eps(f^N) = \int_{\RR^d} (f^N \ast \psi_\eps) \log(f^N \ast \psi_\eps) \rd v
			\end{equation}
			be the discrete entropy, then $\frac{d}{dt} E^N_\eps = - D^N_\eps \leq 0$, where 
			\[
			D_\eps^N = \frac{1}{2}\sum_{i,j} w_{i} w_j \left( \nabla \frac{\delta E_\eps^N}{\delta f} (v_i) - \nabla \frac{\delta E_\eps^N}{\delta f} (v_j)\right)  \cdot A(v_i - v_j) \left( \nabla \frac{\delta E_\eps^N}{\delta f} (v_i) - \nabla \frac{\delta E_\eps^N}{\delta f} (v_j)\right)\,.
			\]
		\end{itemize}
	\end{theorem}
	\begin{proof}
		First, we notice from \eqref{vt2} that
		\begin{align*}
		\frac{d}{dt} \sum_i w_i \phi (v_i) &= \sum_i  w_i \nabla \phi(v_i) \cdot U_\eps (f^N) (t,v_i(t)) \nonumber
		\\ & =  - \sum_{i,j} w_i w_j A(v_i - v_j) \left( \nabla \frac{\delta E_\eps^N}{\delta f} (v_i) - \nabla \frac{\delta E_\eps^N}{\delta f} (v_j)\right) \cdot \nabla \phi(v_i)   \nonumber
		\\ & = -\frac{1}{2} \sum_{i,j} w_i w_j A(v_i - v_j) \left( \nabla \frac{\delta E_\eps^N}{\delta f} (v_i) - \nabla \frac{\delta E_\eps^N}{\delta f} (v_j)\right) \cdot  (\nabla \phi(v_i)   - \nabla \phi(v_j))
		\end{align*}
		which vanishes with $\phi(v) = 1, v, |v|^2$. Therefore, mass, momentum, and energy are preserved. Next, using \eqref{fN}, we rewrite \eqref{EN23} as
		\[
		E_\eps^N = \int_{\RR^d} \left( \sum_i w_i \psi_\eps (v-v_i(t)) \right)  \log \left( \sum_k w_k \psi_\eps (v-v_k(t) )\right) \rd v\,,
		\]
		then 
		\begin{align*}
		\frac{d}{dt} E_\eps^N =& \int_{\RR^d}  \sum_i w_i \nabla \psi_\eps (v-v_i(t)) \frac{dv_i(t)}{dt} \log \left( \sum_k w_k \psi_\eps (v-v_k(t))\right) \rd v 
		\\ & +  \int_{\RR^d}   \left( \sum_i w_i \psi_\eps (v-v_i(t)) \right) \frac{\sum_k  w_k \nabla \psi_\eps (v-v_k(t)) \frac{d v_k(t)}{dt}}{  \sum_k w_k \psi_\eps (v-v_k(t)) }\, \rd v
		\\  =: & I_1 + I_2\,.
		\end{align*}
		Note that $I_2$ can be simplified to
		\begin{align*}
		I_2 = \int_{\RR^d} \sum_k  w_k \nabla \psi_\eps (v-v_k(t)) \frac{d v_k(t)}{dt} \,\rd v = -\frac{d}{dt} \sum_k w_k \int_{\RR^d} \psi_\eps(v-v_k(t))\, \rd v = 0\,,
		\end{align*}
		thanks to the fact that $\int_{\RR^d} \psi_\eps(v-v_k(t)) \rd v  =1$. 
		By virtue of \eqref{nabladeltaE}, $I_1$ has the following estimate
		\begin{align*}
		I_1 = \sum_i w_i \left(\int_{\mathbb{R}^d} \nabla \psi_\eps (v-v_i(t)) \log \left( \sum_k w_k \psi_\eps (v-v_k(t) )\right)\, \rd v\right) \frac{ d v_i}{ dt }  
		= \sum_i w_i \nabla \frac{\delta E_\eps^N}{\delta f} (v_i) \frac{d v_i}{dt}\,.
		\end{align*}
		Then using \eqref{vt2}, it becomes
		\begin{align*}
		I_1  & = \sum_i w_i \nabla \frac{\delta E_\eps^N}{\delta f} (v_i) \left[    -\sum_j w_j A(v_i - v_j) \left( \nabla \frac{\delta E_\eps^N}{\delta f} (v_i) - \nabla \frac{\delta E_\eps^N}{\delta f} (v_j)   \right)  \right]
		\\ &  = - \frac{1}{2}\sum_{i,j} w_{i} w_j \left( \nabla \frac{\delta E_\eps^N}{\delta f} (v_i) - \nabla \frac{\delta E_\eps^N}{\delta f} (v_j)\right)  \cdot A(v_i - v_j) \left( \nabla \frac{\delta E_\eps^N}{\delta f} (v_i) - \nabla \frac{\delta E_\eps^N}{\delta f} (v_j)\right) \leq 0\,,
		\end{align*}
		and therefore, the entropy dissipation follows. 
	\end{proof}
	
		\begin{remark}
			A natural question is how to deal in practice with the cutoff of the initial data in a bounded domain. We will restrict to a square domain $[-L,L]^d$ with $L>0$ as the computational domain due to physical considerations of the Landau equation since the variable is a velocity vector. Notice that the regularized equation \eqref{Landaureg} has the structure of a nonlineary continuity equation in the velocity variable for the unknown density function $f$ with velocity field $U_\eps(f)$. Such continuity equations are naturally posed in a bounded domain with no-flux boundary contions $U_\eps(f)\cdot\nu=0$ at the boundary of the domain where $\nu$ is the outwards unit normal vector to the boundary. This no-flux boundary condition inmmediately leads to mass conservation. In particular, we could solve \eqref{Landaureg} in the square domain $[-L,L]^d$ with no-flux boundary conditions. The particle approximation in \eqref{fN} remains valid and particles follow the same paths as written in \eqref{vt2} as long as the particles do not touch the boundary of the domain. When touching the boundary, particles need to be reflected with respect to the normal direction to impose the no-flux boundary conditions. In the case of the regularized Landau equation, this is not physical since boundaries in the velocity variable do not make sense and because the conservations of mean velocity and energy would be lost when particles are reflected at the boundary of the velocity domain. Therefore, the solutions constructed from particle approximations remain valid as an approximation of the Landau equation as soon as the particles do not touch the boundary. In practice, we initialize with particles chosen in a square domain $[-L,L]^d$ and check that particles do not escape from the domain for their time span to choose the right initialization domain size $L>0$ for our initial data.
		\end{remark}
	
	In practical implementation of particle methods, the update of particle velocity via \eqref{vt2} will not be computed exactly, but with the integral replaced by a quadrature rule. Therefore, we need to introduce a discrete-in-velocity particle method. The computational domain in any dimension is the square domain $[-L,L]^d$ with $L>0$. The mesh size is defined by $h=2L/n$ and $N=n^d$ is the total chosen number of particles. Let us denote the squares of the mesh as $Q_i$ with $i=1,\dots,n^d$. We will always initialize our particle method by projecting the mass of the initial data on the computational domain to a sum of Dirac Deltas located at the center of each $Q_i$ with mass given by the mass of the initial data in $Q_i$, that is
	$$
	\bar f^N(0,v) := \sum_{i = 1}^N w_i \delta(v-\bar{v}_i(0))\,,
	\quad \mbox{with } \bar{v}_i(0)=v_i^c  \mbox{ and }  w_i=f_0(v_i^c)h^d\,,
	$$
	with $v_i^c$ denoting the center of the square $Q_i$. Now, we can introduce the discrete in velocity particle method as
	\[
	\bar{f}^N = \sum_{i=1}^N w_i \delta (v-\bar{v}_i(t)) 
	\]
	where $\bar{v}_i(t)$ satisfies
	\begin{align} \label{vt3}
	\frac{d \bar v_i(t)}{dt}  =& -  \sum_j w_j A(\bar v_i - \bar v_j) \left\{   \sum_{l} h^d \nabla \psi_\eps ( \bar v_i - v_l^c ) \log \left( \sum_k w_k\psi_{\varepsilon}( v_l^c - \bar v_k) \right)  \right. \nonumber
	\\ & \hspace{3.2cm}  - \left.   \sum_l  h^d \nabla \psi_\eps (\bar v_j -  v_l^c) \log \left( \sum_k w_k\psi_{\varepsilon}( v_l^c -\bar v_k) \right)  \right\} \nonumber
	\\ =:&  \,-  \sum_j w_jA(\bar v_i - \bar v_j) \left[   \bar F_\eps^N (\bar v_i) - \bar F_\eps^N ( \bar v_j)  \right] =: \,\bar U_\eps (\bar f^N) (t, \bar v_i(t))\,.
	\end{align}
	Here, the function 
	\begin{equation} \label{nabladeltaEdd}
	\bar{F}_\eps^N(\bar v_i) = \sum_l h^d \nabla \psi_\eps(\bar v_i - v_l^c) \log \left(\sum_k w_k \psi_\eps (v_l^c- \bar v_k) \right) 
	\end{equation}
	is a discrete analogue of the first variation of the entropy functional \eqref{nabladeltaE}. One can also define the fully discrete regularized entropy as 
	\begin{equation}\label{discreteentropy}
	\bar E_\eps^N  = \sum_l h^d \left( \sum_i w_i \psi_\eps(v_l^c - \bar v_i ) \right) \log\left(  \sum_k w_k \psi_\eps(v_l^c - \bar v_k )\right) \,.
	\end{equation}
	Then we show that at this fully discrete level, some properties in Theorem \ref{thm10} are inherited.  
	
	\begin{theorem}
		The discrete-in-velocity particle method \eqref{vt3} satisfies the conservation of mass, momentum, and energy. Moreover, the discrete entropy \eqref{discreteentropy} almost decays in time, that is, 
		$$
		\bar E^N_\eps (t) - \bar E^N_\eps(0) = - \int_0^t\bar D^N_\eps ds +O(h^2)\,, 
		$$
		where 
			\[
			\bar D_\eps^N = \frac{1}{2}\sum_{i,j} w_{i} w_j \left( \bar{F}_\eps^N ( \bar v_i) - \nabla \bar{F}_\eps^N ( \bar v_j)\right)  \cdot A( \bar v_i -  \bar v_j) \left( \bar{F}_\eps^N ( \bar v_i) - \bar{F}_\eps^N ( \bar v_j)\right)\geq 0\,.
			\]
	\end{theorem}
	\begin{proof}
		Indeed, for $\phi (v) = 1, v , |v|^2$, we have
		\begin{align}
		\frac{d}{dt} \sum_i w_i \phi (\bar v_i) &= \sum_i  w_i \nabla \phi(\bar v_i) \cdot \bar U_\eps (f^N) (t,\bar v_i(t)) \nonumber
		\\ & =  - \sum_{i,j} w_i w_j A( \bar v_i -  \bar v_j) \left( \bar F_\eps^N (\bar v_i) - \bar F_\eps^N (\bar v_j)\right) \cdot \nabla \phi( \bar v_i)   \nonumber
		\\ & = -\frac{1}{2} \sum_{i,j} w_i w_j A( \bar v_i - \bar v_j) \left( \bar F_\eps^N (\bar v_i) - \bar F_\eps^N (\bar v_j)\right) \cdot  (\nabla \phi( \bar v_i)   - \nabla \phi( \bar v_j)) = 0\,, \label{conservdis}
		\end{align}
		hence the conversation of mass, momentum, and energy is guaranteed. A similar computation to the entropy dissipation in the semidiscrete level leads to
		\begin{align*}
		\frac{d}{dt} \bar E_\eps^N(t) =& \sum_l h^{d} \sum_i w_i \nabla \psi_\eps (v_l^c-\bar v_i(t)) \frac{d\bar v_i(t)}{dt} \log \left( \sum_k w_k \psi_\eps (v_l^c-\bar v_k(t))\right) 
		\\ & +  \sum_l  h^{d} \left( \sum_i w_i \psi_\eps (v_l^c-\bar v_i(t)) \right) \frac{\sum_k  w_k \nabla \psi_k (v_l^c-\bar v_k(t)) \frac{d \bar v_k(t)}{dt}}{  \sum_k w_k \psi_\eps (v_l^c-\bar v_k(t)) } 
		\\  =: & I_1 + I_2\,.
		\end{align*}
		By the definition of \eqref{nabladeltaEdd} and similarly to \eqref{conservdis}, $I_1$ can be written as 
		\begin{align*}
		I_1 = \sum_i w_i \bar{F}_\eps^N (\bar v_i) \frac{d \bar v_i}{dt}=-\frac{1}{2} \sum_{i,j} w_i w_j A( \bar v_i - \bar v_j) \left( \bar F_\eps^N (\bar v_i) - \bar F_\eps^N (\bar v_j)\right) \cdot  (\bar{F}_\eps^N( \bar v_i)   - \bar{F}_\eps^N( \bar v_j)) \leq 0\,.
		\end{align*}
		As before $I_2$ can be written as
		\begin{align*}
		I_2 &= \sum_l h^{d} \sum_i  w_i \nabla \psi_\eps (v_l^c-\bar v_i(t)) \frac{d\bar v_i(t)}{dt} = \frac{d}{dt} \sum_i w_i \sum_l h^{d} \psi_\eps(v_l^c-\bar v_i(t))\,.
		\end{align*}
		We are reduced to showing that 
			$$ 
			\sum_i w_i \sum_l h^{d} \psi_\eps(v_l^c-\bar v_i(t))= \sum_i w_i + O(h^2)
			$$
			which is true thanks to the fact that $\int_{\RR^d} \psi_\eps(v-\bar v_k(t)) \rd v  =1$ and that the mid-point composite quadrature rule is of order 2 for smooth functions. Note that the constant in the error depends on $\eps$ but not on the location of the particles. Therefore, we conclude that
			$$
			\int_0^t I_2 \,ds = \sum_i w_i \sum_l h^{d} \psi_\eps(v_l^c-\bar v_i(t)) - \sum_i w_i \sum_l h^{d} \psi_\eps(v_l^c-\bar v_i(0)) = O(h^2)
			$$
			in the time interval $[0,t]$.
	\end{proof}
	
	\begin{remark}
		The particle method for the alternative regularization for the entropy \eqref{alter} has the advantage of not needing a continuous convolution and it also has the conservation and dissipative properties. The particle method reads as
		\begin{align} \label{vt4}
		\frac{d \tilde v_i(t)}{dt} =  -  \sum_j w_jA( \tilde v_i -  \tilde v_j) \left[  \nabla \frac{\delta \tilde E^N_\eps}{\delta f}( \tilde v_i) - \nabla \frac{\delta  \tilde E^N_\eps}{\delta f}( \tilde v_j)  \right] \,,
		\end{align}
		with
		\begin{align*}
		\nabla \frac{\delta \tilde E^N_\eps}{\delta f}(v)  = &\, \frac{ \sum_k w_k \nabla\psi_{\varepsilon}( v-\tilde v_k)}{ \sum_k w_k\psi_{\varepsilon}( v-\tilde v_k)} + \sum_k w_k \frac{\nabla\psi_{\varepsilon}( v-\tilde v_k)}{ \sum_m w_m \psi_{\varepsilon}(\tilde v_k-\tilde v_m)}\,,
		\end{align*}
		according to \eqref{altervar}. One can show that the semidiscrete particle method \eqref{vt4} satisfies the conservation of mass, momentum, and energy and the dissipation of entropy defined as
		\begin{equation*}
		\tilde E^N_\eps =\sum_i w_i \log\left(  \sum_j w_j \psi_\eps(\tilde v_i - \tilde v_j )\right)\,,
		\end{equation*}
		then $\frac{d}{dt} \tilde E^N_\eps = - \tilde D^N_\eps \leq 0$, where 
		\[
		\tilde D_\eps^N = \frac{1}{2}\sum_{i,j} w_{i} w_j \left( \nabla \frac{\delta \tilde E_\eps^N}{\delta f} (\tilde v_i) - \nabla \frac{\delta \tilde E_\eps^N}{\delta f} (\tilde v_j)\right)  \cdot A(\tilde v_i -\tilde  v_j) \left( \nabla \frac{\delta \tilde E_\eps^N}{\delta f} (\tilde v_i) - \nabla \frac{\delta \tilde E_\eps^N}{\delta f} (\tilde v_j)\right)\,.
		\]
		This alternative regularization will be explored elsewhere.
	\end{remark}
	
	
	\section{Numerical implementation and simulation}
	
	In order to visualize our particle solution and compare it to the exact solutions in smoother norms, we construct a blob solution, as in \cite{carrillo2017blob}, obtained by convolving the particle solution with the mollifier,
	\begin{align} \label{blobsolution}
	\tilde{f}^N(t,v) := (\psi_\eps*\bar f^N)(t,v)= \sum_{i = 1}^N w_i \psi_\eps(v-\bar v_i(t))\,,
	\end{align}
	with $\bar v_i(t)$ given by \eqref{vt3} for all $t>0$. We measure the accuracy of our numerical method with respect to the $L^1$- and $L^\infty$-norms. To compute the $L^1$- and $L^\infty$-errors, we take the difference between the exact or reference solution and the blob solution \eqref{blobsolution} and evaluate discrete $L^p$- and $L^\infty$-norms in a grid. The norms will be computed in this computational mesh using the centers of the squares $Q_i$ as
	$$
	\|g\|_{L^{p}}^p = \sum_{i=1}^N h^d |g(v_i^c)|^p\,, \quad \|g\|_{L^{\infty}} = \max_{i} |g(v_i^c)|\,,
	$$
	for any function $g$ defined on the computational mesh, and $1\leq p < \infty$. The quantities of interest will be computed as follows: the discrete mass, momentum and energy are defined as
	\begin{equation*}
	\sum_{i=1}^N w_i ,\quad \sum_{i=1}^N w_i \bar v_i \quad \mbox{and} \quad\sum_{i=1}^N w_i |\bar v_i|^2\,,
	\end{equation*}
	respectively. The discrete entropy is defined by $\bar E_\eps^N$ in \eqref{discreteentropy}.
	
	Let us now comment on the practical implementation of the method. The time discretization of the system of ODEs defined by the particle method \eqref{vt3} is done by the simple explicit Euler method. This choice is motivated by our main purpose: we want to illustrate the performance of this particle method by focusing on the basic properties and its capabilities even with the lowest order in time discretization. Note that the fully discrete-in-time method conserves mass and momentum exactly, but the energy conservation is satisfied up to a first order error in time. Indeed, mass is automatically conserved. To see the momentum conservation, note that a time discrete version of \eqref{vt2} yields
		\begin{align}
		\frac{1}{\Delta t} (v_i^{n+1}-v_i^n)& =   -  \sum_j w_jA(v_i^n - v_j^n) \left[  \nabla \frac{\delta E^N_\eps}{\delta f}(v_i^n) - \nabla \frac{\delta E^N_\eps}{\delta f}(v_j^n)  \right].
		\end{align}
		Multiplying both hand sides by $w_i$ and sum over $i$, we obtain
		\begin{align}
		\frac{1}{\Delta t} (\sum_i w_i v_i^{n+1}-\sum_i w_i v_i^n)& =   -  \sum_{ij} w_iw_jA(v_i^n - v_j^n) \left[  \nabla \frac{\delta E^N_\eps}{\delta f}(v_i^n) - \nabla \frac{\delta E^N_\eps}{\delta f}(v_j^n)  \right]\nonumber\\
		&=   \sum_{ij} w_iw_jA(v_i^n - v_j^n) \left[  \nabla \frac{\delta E^N_\eps}{\delta f}(v_i^n) - \nabla \frac{\delta E^N_\eps}{\delta f}(v_j^n)  \right]=0,
		\end{align}
		where in the second equality, we switched $i$ and $j$ and used symmetry of matrix $A$. However, the same trick does not work in the energy case, hence the energy is only conserved up to $O(\Delta t)$. The numerical example in the next section (in particular, Figure~\ref{fig:2DBKW_energy} (left)) also confirms this fact.
	
	We will check these issues later on in the examples. One can obviously improve some of the time discretization errors committed by choosing higher order time approximations of the ODE system with adaptive time stepping. We leave this for future work in the scientific computing direction focusing here on the convergence analysis and error in velocity of the particle approximation \eqref{vt3}.
	
	As usual in particle methods, the regularization parameter has to be chosen very carefully. This regularization was already used for nonlinear diffusion and aggregation-diffusion equations in \cite{carrillo2017blob}. It was proven in \cite[Theorem 6.1]{carrillo2017blob} that, for the porous medium equation with exponent larger than or equal to 2, a particle method using the regularization strategy presented in this work is convergent by choosing $h^2= o(\eps)$ as $\eps\to 0$. By choosing $h^p\simeq\eps$, the previous constraint is satisfied for $0<p<2$. Then, it was checked heuristically that with $\eps\simeq h^{1.98}$, the numerical particle scheme is a second order approximation to the solutions of all nonlinear degenerate diffusion equations of porous medium type and also for the heat equation. Notice it is more convenient to choose the largest possible $h$ to have the least number of particles since $h=2L/n$. For these reasons, the regularizing parameter for the Landau equation is chosen as 
	$\varepsilon=0.64h^{1.98}$. Here the prefactor is empirical and is found by trial and error. 
	
	Finally, let us comment that this error estimate is different for transport equations as studied in \cite{CP,CB}. For the transport equation, depending on the regularity of the initial data, one gets $h^p\simeq\eps$ for $0<p<1$, that is $h= o(\eps)$ meaning that for transport equations one needs typically smaller meshes and therefore more particles than for diffusion-type equations.

	\subsection{Example 1: 2D BKW solution for Maxwell molecules}
	
	In this and next subsections, we use the BKW solution in 2D and 3D to validate the accuracy of our method. This is one of the few analytical solutions one can construct for the Landau equation. For the reader's convenience, we give the derivation in Appendix A.
	
	Consider the collision kernel
	\begin{equation*}
	A(z)=\frac{1}{16}(|z|^2I_d-z\otimes z),
	\end{equation*}
	and an exact solution given by
	\begin{equation*}
	f^{\text{ext}}(t,v)=\frac{1}{2\pi K}\exp \left ( -\frac{|v|^2}{2K}\right)\left(\frac{2K-1}{K}+\frac{1-K}{2K^2}|v|^2\right), \quad K=1-\exp(-t/8)/2.
	\end{equation*}
	We choose $t_0=0$ and compute the solution until $t=5$. The number of particles are chosen as $N=n^2$ with $n=60, 80, 100, 120, 150$. The computational domain is $[-L,L]^2$ with $L=4$, so the initial mesh size is $h=2L/n$. The forward Euler method with $\Delta t =0.01$ is used for time discretization.
	
	We first track the relative $L^2$ error of the solution, see Figure~\ref{fig:2DBKW_error} (left), from which we observe the errors remain stable over time and decrease with higher number of particles. To check the decay rate, we generate the loglog plot of the errors at a fixed time $t=5$, see Figure~\ref{fig:2DBKW_error} (right). Here the $x$-axis is $h$, i.e., the initial mesh size. Using the least square fitting, we can find the approximate slope of the errors which exhibits almost second order convergence. 
	
	\begin{figure}[htp]
		\begin{center}
			\includegraphics[width=3.2in,height=2.5in]{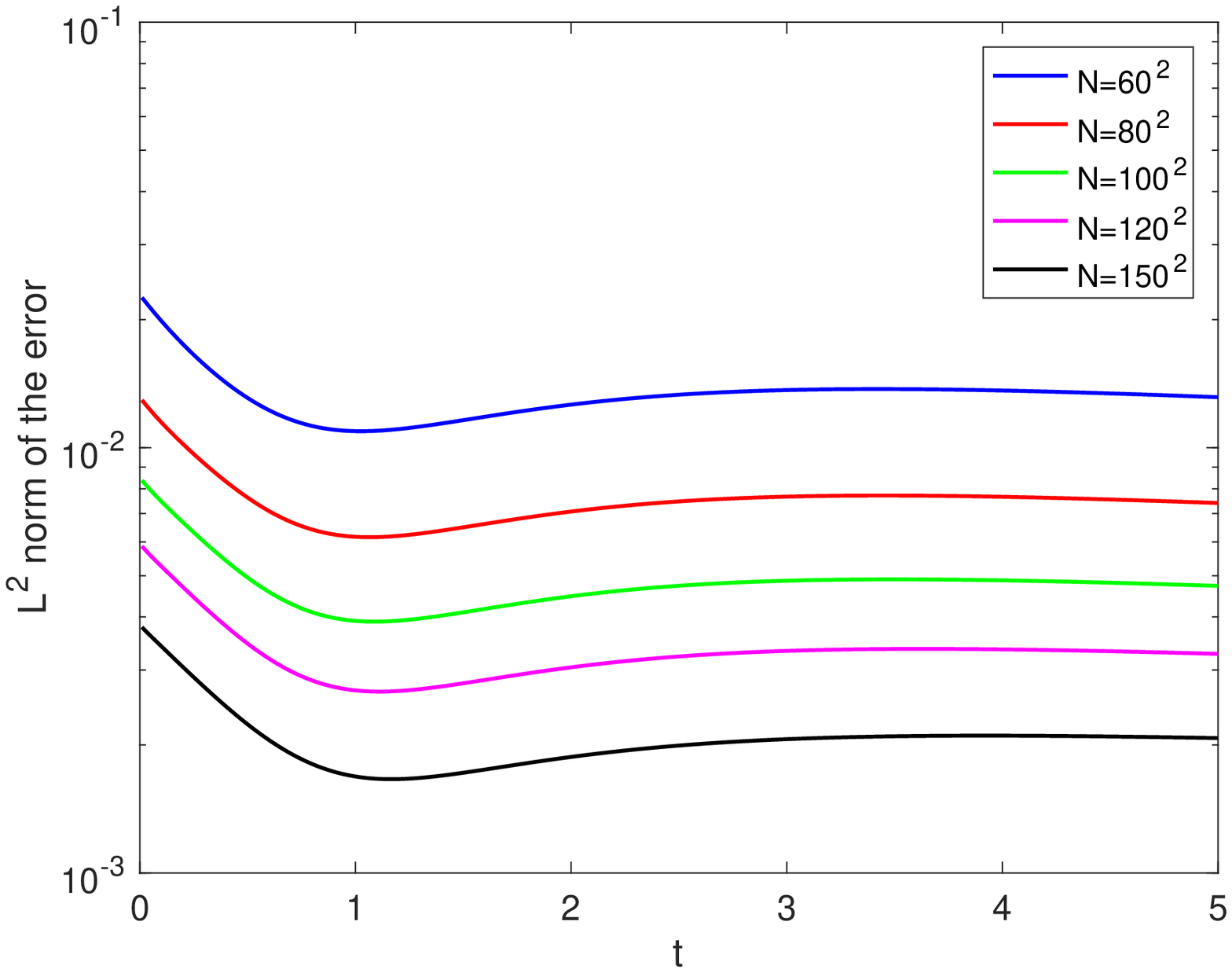}
			\includegraphics[width=3.2in,height=2.5in]{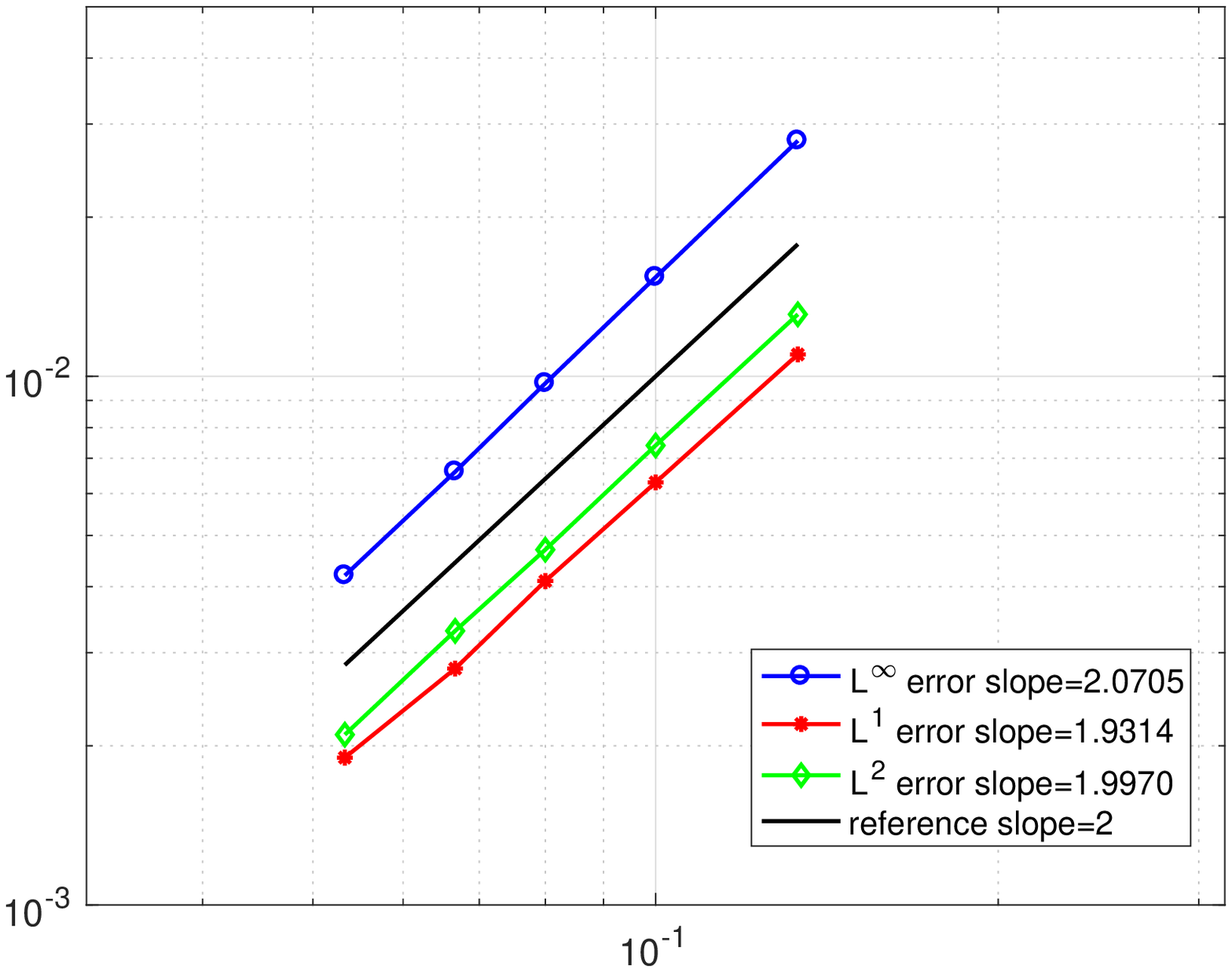}
		\end{center}
		\caption{Left: Time evolution of $\|f^{\text{num}}-f^{\text{ext}}\|_{L^{2}}/\|f^{\text{ext}}\|_{L^{2}}$ with respect to different number of particles. Right: Relative $L^{\infty}$, $L^1$, and $L^2$ norms of the error at time $t=5$ with respect to different $h$.}
		\label{fig:2DBKW_error}
	\end{figure}
	
	To further check the conservation and entropy decay properties of the method, we plot the time evolution of the total energy and relative entropy of the system in Figure~\ref{fig:2DBKW_energy}. The energy is conserved up to a very small error (this error decays when the time step decreases) while the entropy decays monotonically as expected.
	Analogously to equation \eqref{discreteentropy}, we define the relative entropy as
	\begin{equation*}
	\sum_l h^d   \left( \sum_{k=1}^N w_k\psi_{\varepsilon}(v_l^c-\bar{v}_k)\right)\left( \log \left( \sum_{k=1}^N w_k\psi_{\varepsilon}(v_l^c-\bar{v}_k)\right)  +  \log (2\pi)   + \frac{1}{2}  |v_l^c|^2 \right)\,.
	\end{equation*}

	\begin{figure}[htp]
		\begin{center}
			\includegraphics[width=3.2in,height=2.5in]{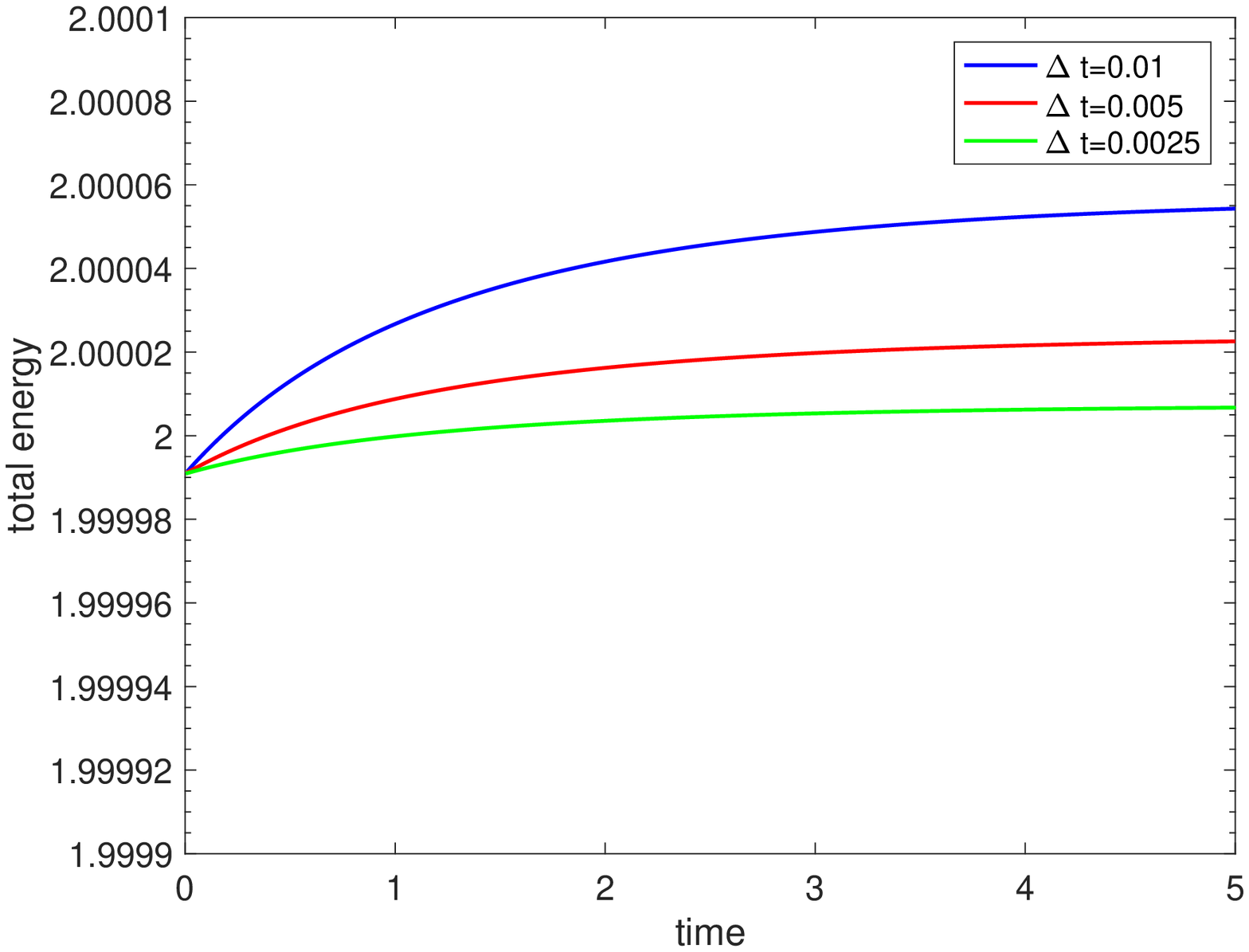}
			\includegraphics[width=3.2in,height=2.5in]{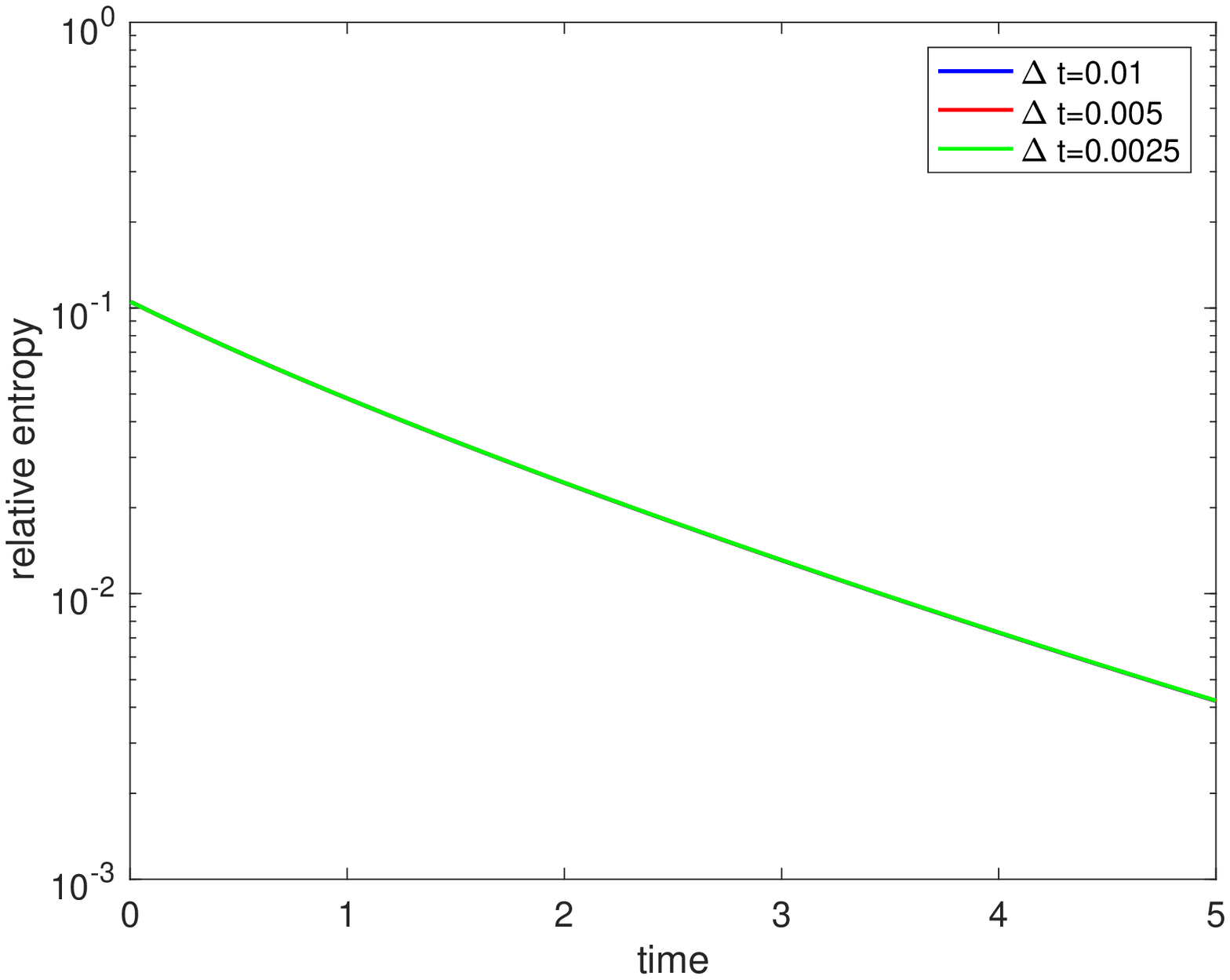}
		\end{center}
		\caption{Time evolution of the total energy (left) and relative entropy (right) with respect to different time step. Particle number $N=60^2$ is fixed.}
		\label{fig:2DBKW_energy}
	\end{figure}


	\subsection{Example 2: 3D BKW solution for Maxwell molecules}
	
	Consider the collision kernel
	\begin{equation*}
	A(z)=\frac{1}{24}(|z|^2I_d-z\otimes z),
	\end{equation*}
	and an exact solution given by
	\begin{equation*}
	f^{\text{ext}}(t,v)=\frac{1}{(2\pi K)^{3/2}}\exp \left ( -\frac{|v|^2}{2K}\right)\left(\frac{5K-3}{2K}+\frac{1-K}{2K^2}|v|^2\right), \quad K=1-\exp(-t/6).
	\end{equation*}
	We choose $t_0=5.5$ and compute the solution until $t=6$. The number of particles are chosen as $N=n^3$ with $n=20, 30, 40, 50, 60$. The computational domain is $[-L,L]^3$ with $L=4$, so the initial mesh size is $h=2L/n$. The forward Euler method with $\Delta t =0.01$ is used for time discretization.
	
	Here we plot similar figures as in the 2D case. We mention that the direct computation in 3D is computationally costly so that we cannot afford too many particles and the errors are generally larger than in 2D. Remarkably, even with a small number of particles, up to $60^3$, we are still able to observe the second order convergence in $L^1$ and $L^2$ norms ($L^{\infty}$ norm is not very reliable due to the limited number of particles), see Figure~\ref{fig:3DBKW_error}.
	
	\begin{figure}[htp]
		\begin{center}
			\includegraphics[width=3.2in,height=2.5in]{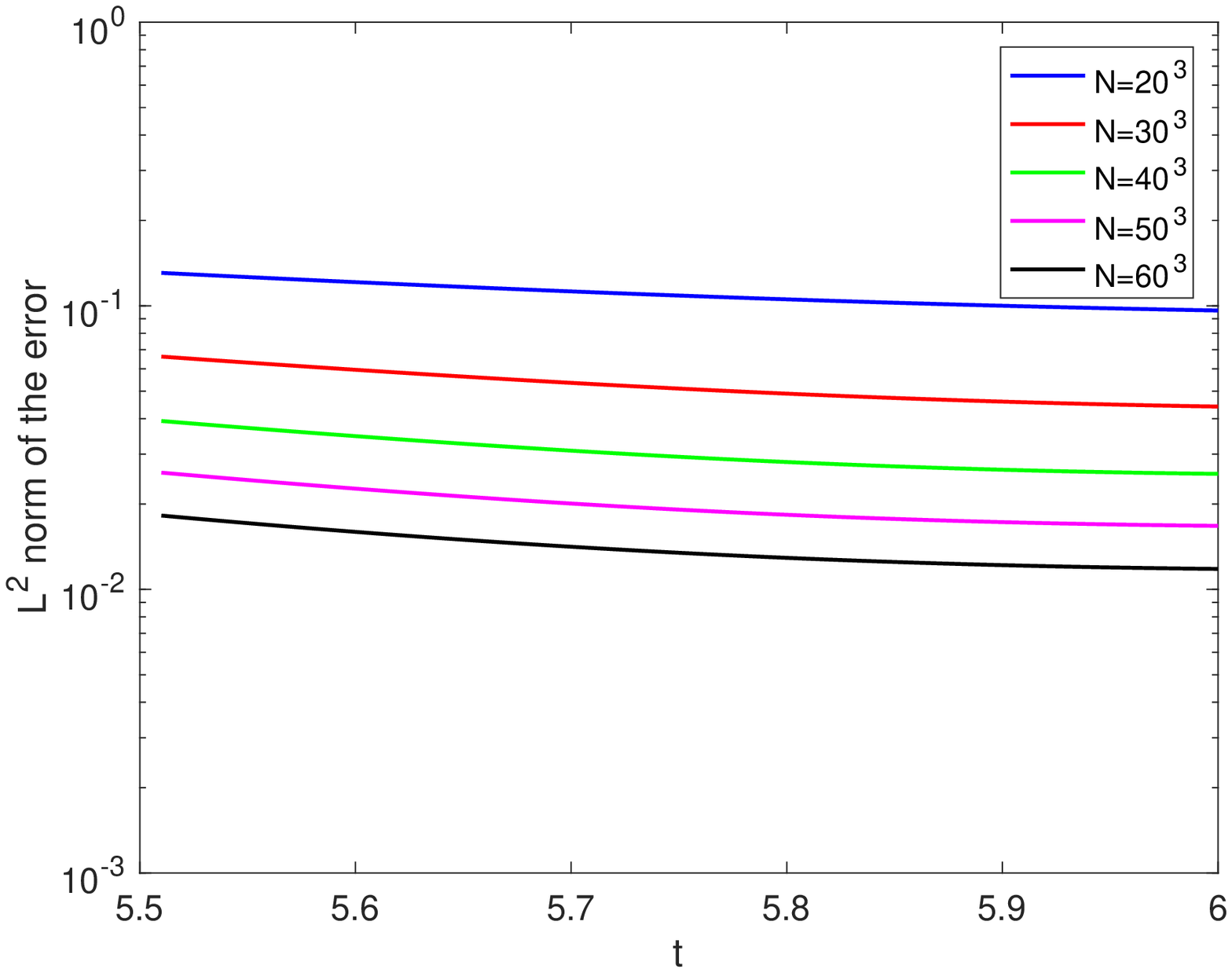}
			\includegraphics[width=3.2in,height=2.5in]{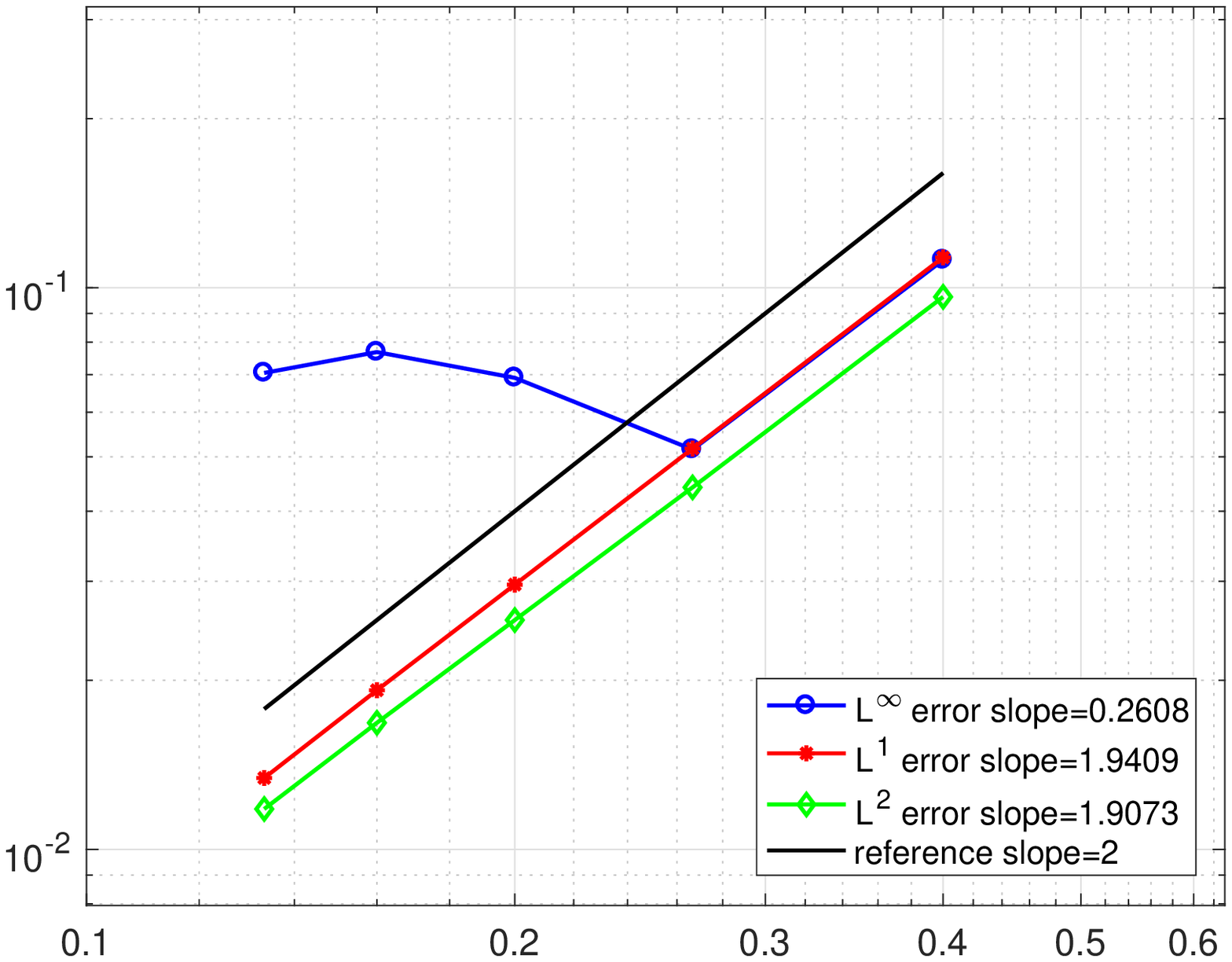}
		\end{center}
		\caption{Left: Time evolution of $\|f^{\text{num}}-f^{\text{ext}}\|_{L^{2}}/\|f^{\text{ext}}\|_{L^{2}}$ with respect to different number of particles. Right: Relative $L^{\infty}$, $L^1$, and $L^2$ norms of the error at time $t=6.5$ with respect to different $h$.}
		\label{fig:3DBKW_error}
	\end{figure}
	

	\subsection{Example 3: 2D anisotropic solution with Coulomb potential}
	
	Consider the collision kernel
	\begin{equation*}
	A(z)=\frac{1}{16}\frac{1}{|z|^3}(|z|^2I_d-z\otimes z),
	\end{equation*}
	and the initial condition
	\begin{equation*}
	f(0,v)=\frac{1}{4\pi}\left\{\exp\left(-\frac{(v-u_1)^2}{2}\right)+\exp\left(-\frac{(v-u_2)^2}{2}\right) \right\}, \quad u_1=(-2,1), \quad u_2=(0,-1).
	\end{equation*}
	For this example, we do not have the exact solution to compare with. Therefore, we compare the particle method with the Fourier spectral method in \cite{PRT00}. For the particle method, we choose the following parameters: the number of particles is $N=120^2$ and the computational domain is $[-10,10]^2$. The forward Euler method with $\Delta t =0.1$ is used for time discretization.
	
	For the spectral method, we choose the following parameters: the number of Fourier modes in each velocity dimension is $N_v=128$; the computational domain is $[-10,10]^2$. The second order Heun's method with $\Delta t=0.1$ is used for time discretization.
	
	The results are shown in Figure
	\ref{fig:2DGaussian}. The results of the two methods match very well.
	
	
	\begin{figure}[htp]
		\begin{center}
			\includegraphics[width=3.2in,height=2.5in]{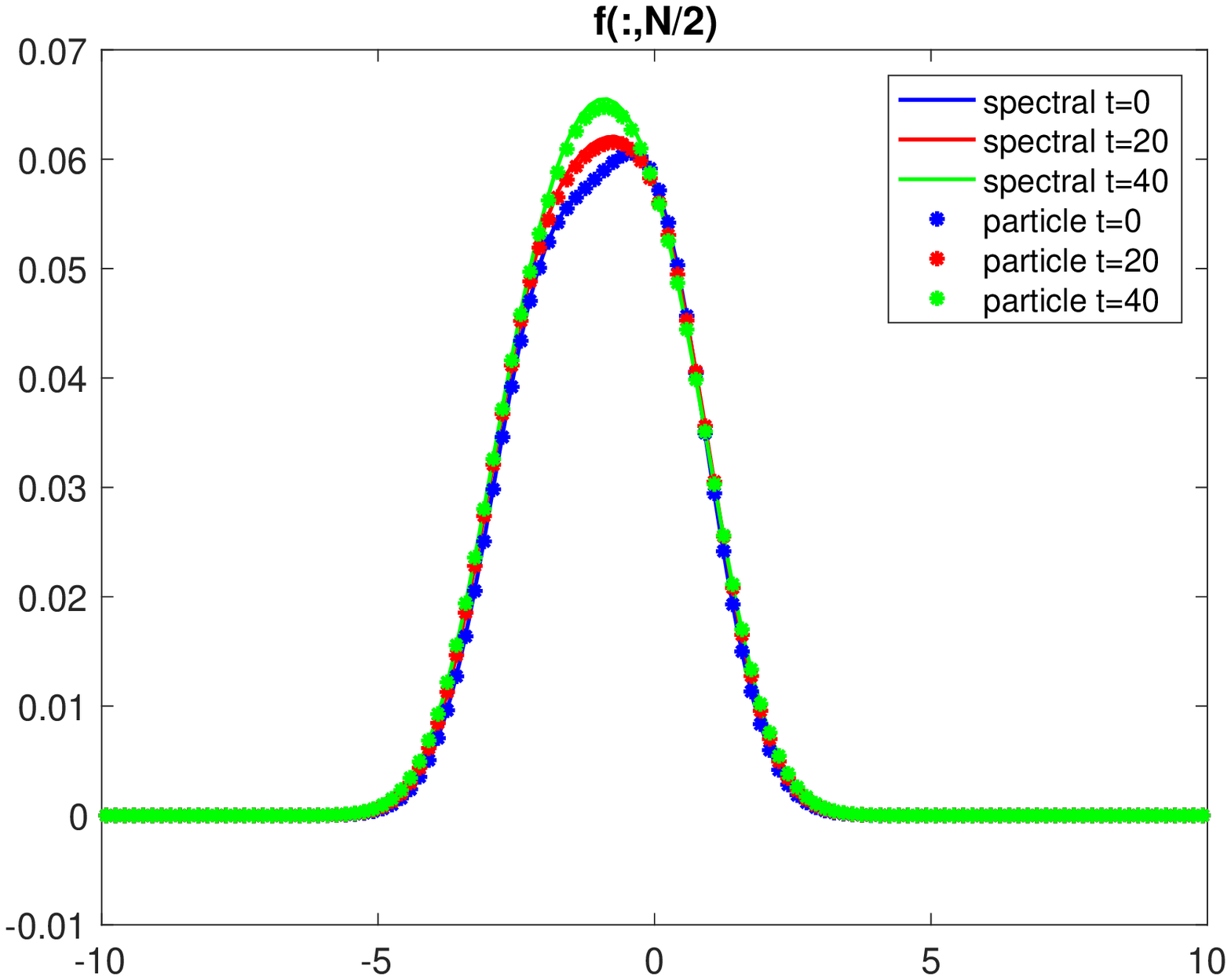}
			\includegraphics[width=3.2in,height=2.5in]{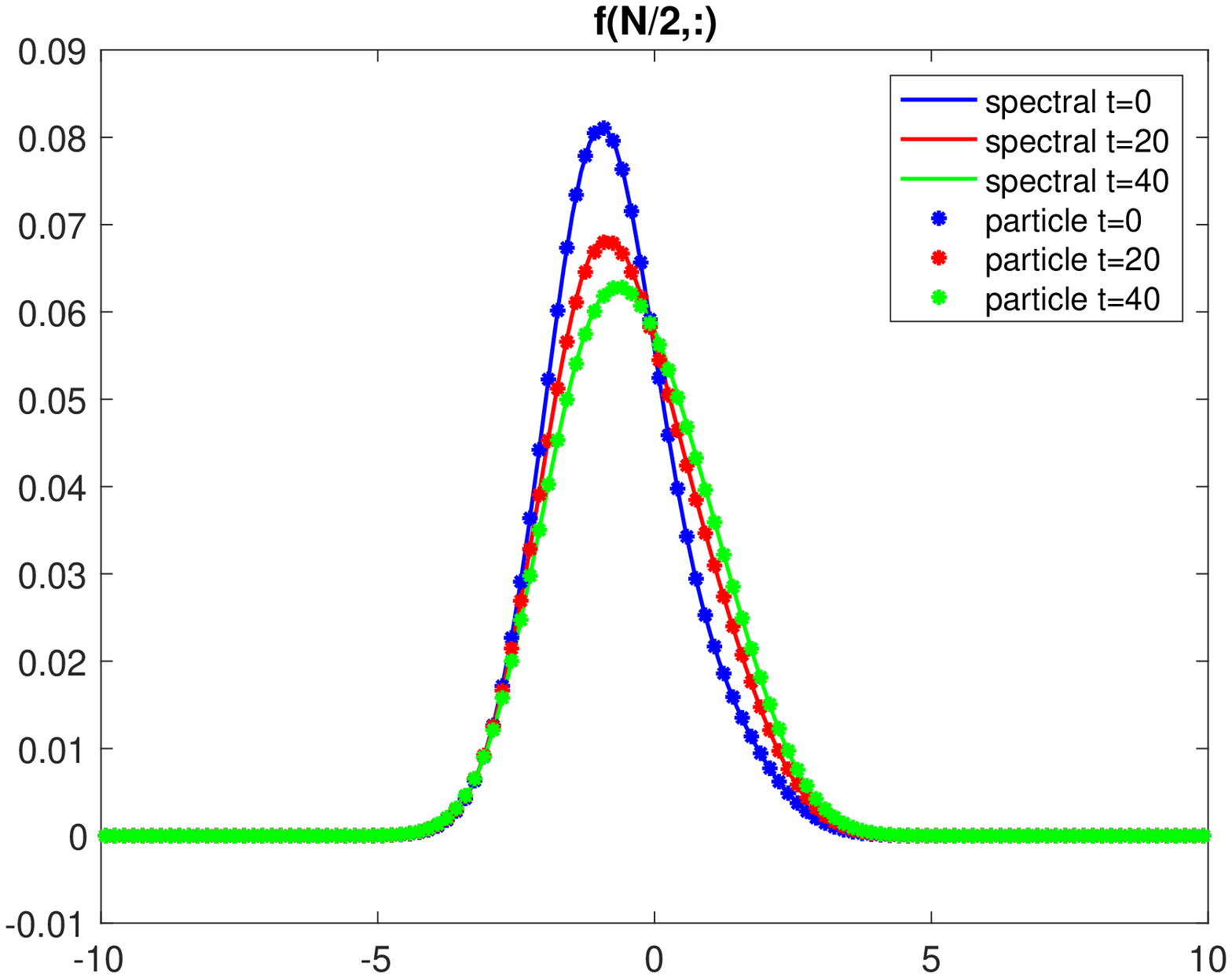}
		\end{center}
		\caption{Comparison of the particle method (particle number $N=120^2$) with the spectral method ($N_v=128^2$). Slices of the solutions at different times.}
		\label{fig:2DGaussian}
	\end{figure}
	
	To better check the convergence of the particle method, we use the spectral method solution with $N_v=128$ as a reference solution. For the particle method, we test $N=60^2, 80^2, 100^2, 120^2$, and for each of them reconstruct the solution on the same mesh as the spectral method (so that we can directly compare the error). The results are shown in Figure~\ref{fig:2DGaussian_order} where we can observe better match as $N$ increases. We also compute the convergence order similarly as in example 1. Strikingly, we can still obtain almost second order convergence, see Figure~\ref{fig:2DGaussian_error}.
	
	\begin{figure}[htp]
		\begin{center}
			\includegraphics[width=3.2in,height=2.5in]{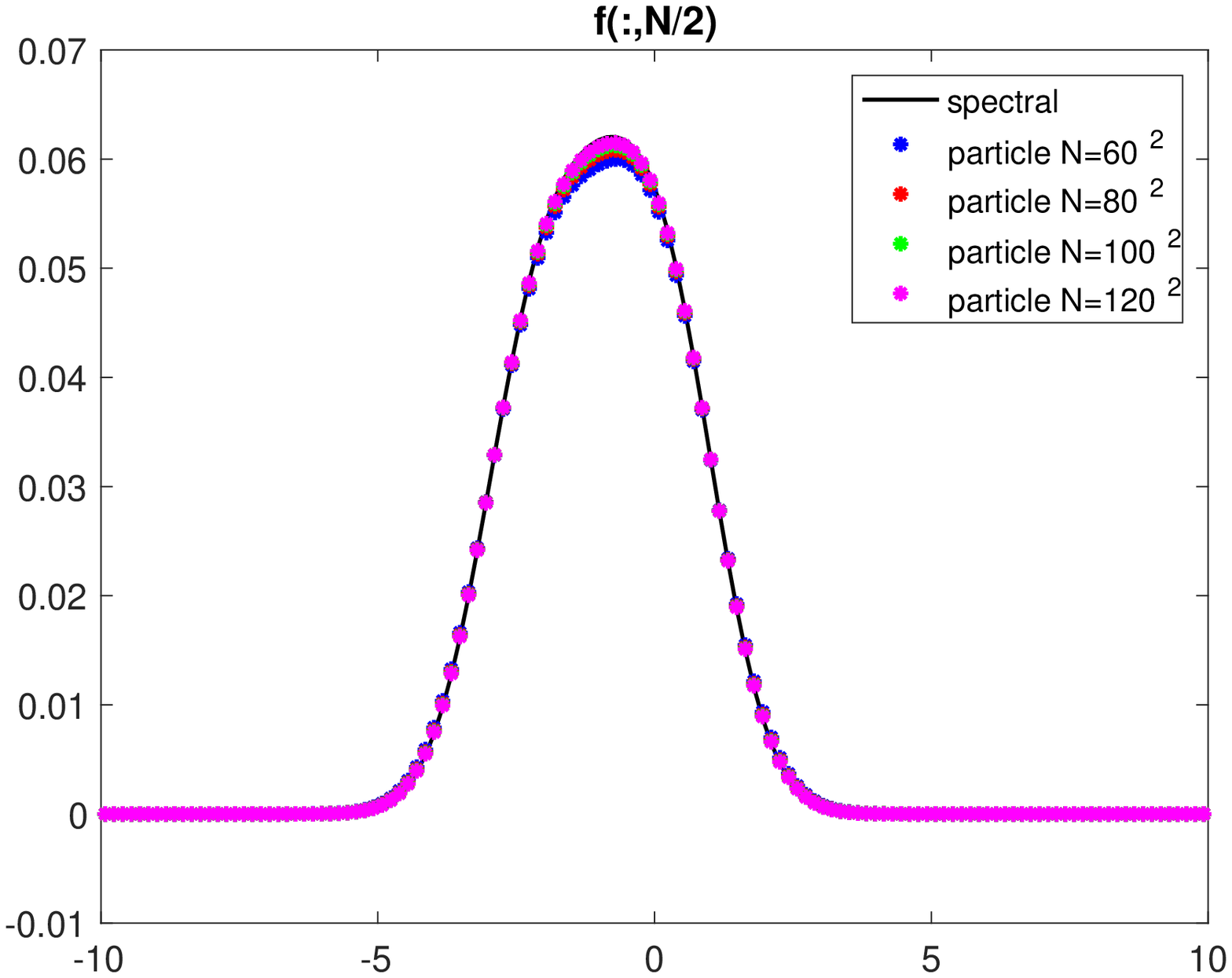}
			\includegraphics[width=3.2in,height=2.5in]{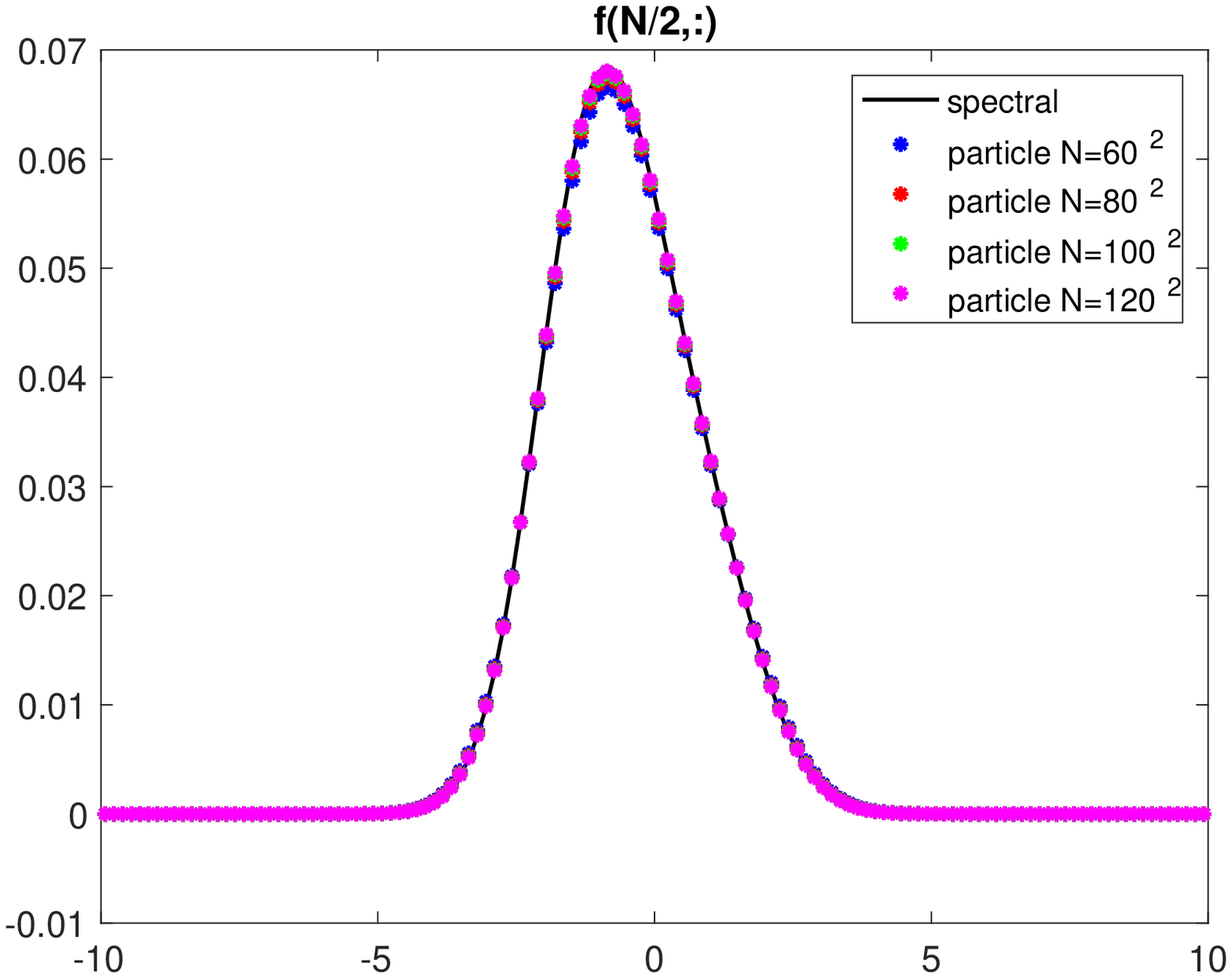}
		\end{center}
		\caption{Comparison of the particle method (using different particle numbers) with the spectral method ($N_v=128^2$). Slices of the solutions at time $t=20$.}
		\label{fig:2DGaussian_order}
	\end{figure}
	
	\begin{figure}[htp]
		\begin{center}
			\includegraphics[width=3.2in,height=2.5in]{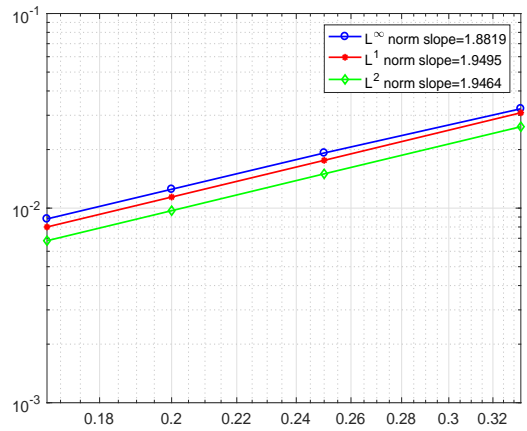}
		\end{center}
		\caption{Relative $L^{\infty}$, $L^1$, and $L^2$ norms of the error at time $t=20$ with respect to different $h$.}
		\label{fig:2DGaussian_error}
	\end{figure}

	\subsection{Example 4: 3D Rosenbluth problem with Coulomb potential}
	
	Consider the collision kernel
	\begin{equation*}
	A(z)=\frac{1}{4\pi}\frac{1}{|z|^3}(|z|^2I_d-z\otimes z),
	\end{equation*}
	and the initial condition
	\begin{equation*}
	f(0,v)=\frac{1}{S^2}\exp\left(-S\frac{(|v|-\sigma)^2}{\sigma^2}\right), \quad \sigma=0.3, \quad S=10.
	\end{equation*}
	A similar test has been considered in other papers \cite{PRT00}. For the particle method, we choose the following parameters: the number of particles is $N=50^3$; the computational domain is $[-1,1]^3$. The forward Euler method with $\Delta t =0.2$ is used for time discretization.
	
	For the spectral method, we choose the following parameters: the number of Fourier modes in each velocity dimension is $N_v=64$; the computational domain is $[-1,1]^3$. The second order Heun's method with $\Delta t=0.2$ is used for time discretization.
	
	The cost of computing the particle method in 3D becomes very heavy if the right-hand side of \eqref{vt3} is performed by direct sums. We resort to efficient methods for computing large sums involving convolution kernels. One possible choice is to make use of the treecode strategy as in \cite{BH86,LJK09} for instance. We give a brief account of its application to the particle method \eqref{vt3} in Appendix \ref{appb}. In Figure \ref{fig:3DRosenbluth} left, we show the comparison of the direct sum solver to the trecode solver by plotting their solutions at $t=20$, $N=50^3$ or $N=40^3$. The error committed is negligible.  In Figure \ref{fig:3DRosenbluth} right, we illustrate the speed-up of the treecode solver with respect to the direct sum solver. The efficiency of the treecode solver is significant with larger number of particles $N$ as expected. The results are obtained using Matlab code on Minnesota Supercomputer Institute Mesabi machine with 12 nodes, further speed up are anticipated with C++ code. 
	
	The result is shown in Figure~\ref{fig:Rosenbluth} which we observe good agreement between the spectral method and the particle method using the treecode acceleration, especially for short time. For longer time, the discrepancy is due to the limited resolution of the particle method. Note that we do get better convergence when increasing the number of particles from $N=50^3$ to $N=60^3$.

	\begin{figure}[htp]
		\begin{center}
			\includegraphics[width=3.2in,height=2.4in]{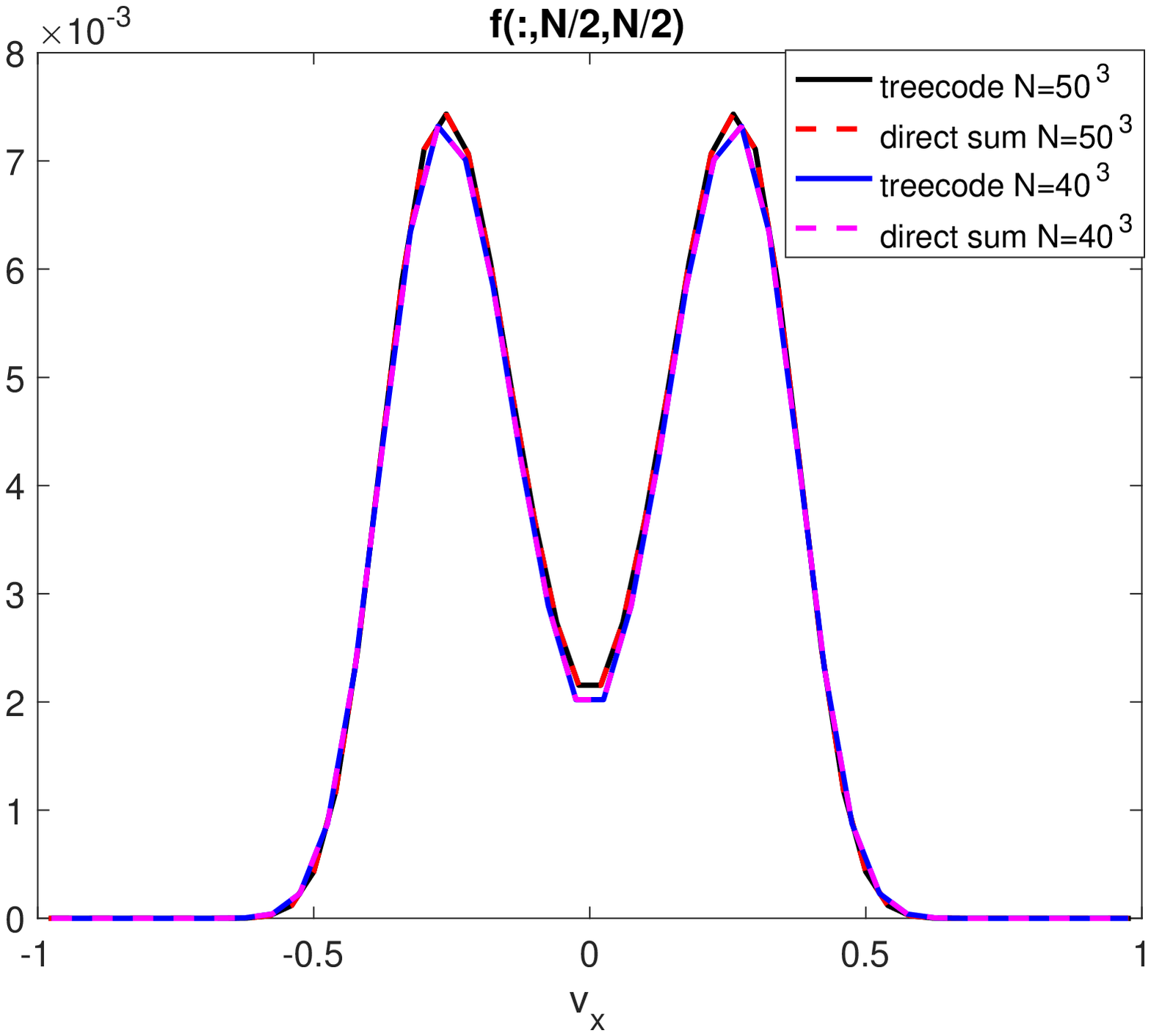}
			\includegraphics[width=3.2in,height=2.4in]{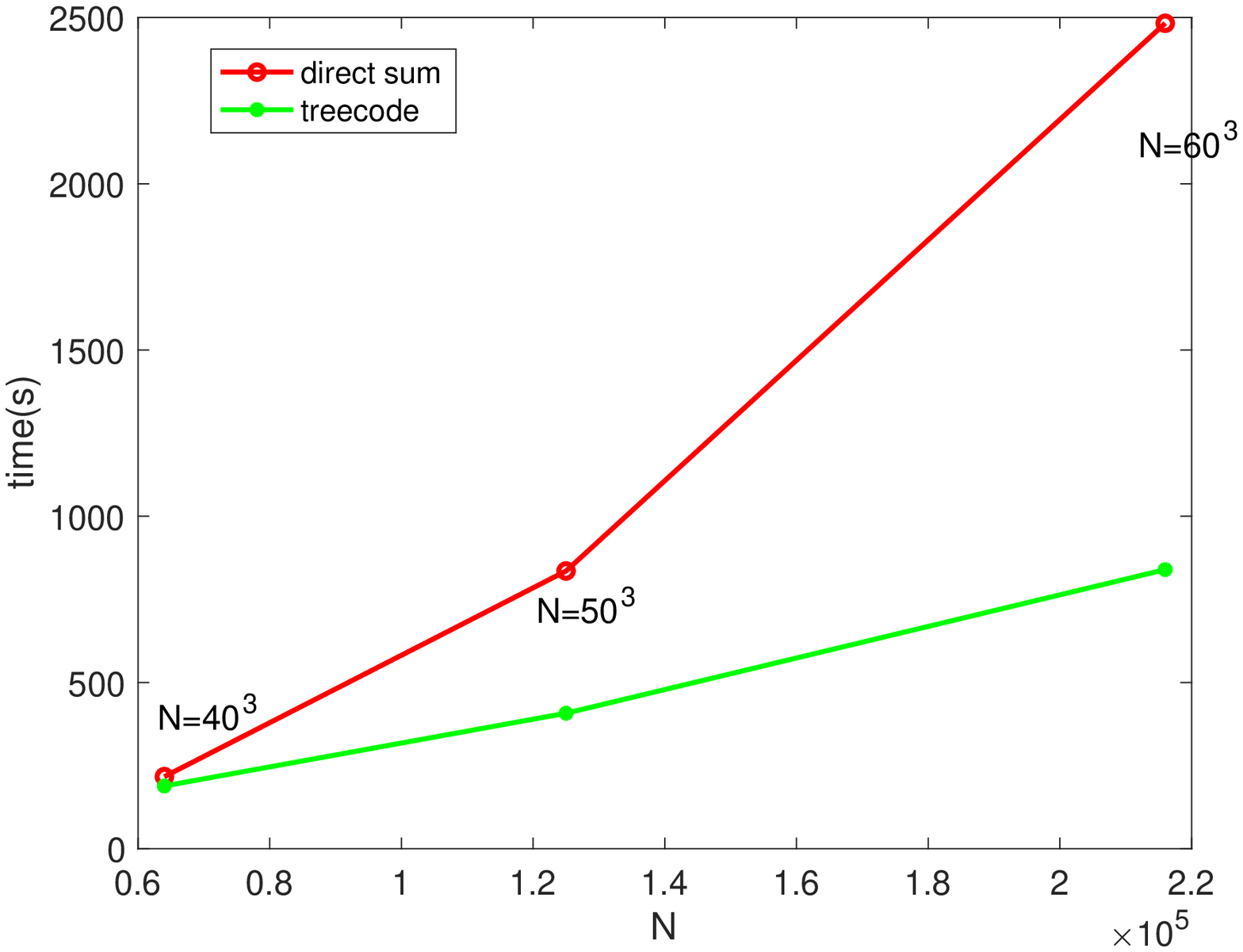}
		\end{center}
		\caption{Left: comparing a slice of the solution with direct sum and treecode at $t=20$, $N=50^3$ or $N=40^3$. Right: comparison of computational time (in seconds) for one step with the treecode solver and with the direct sum solver. }
		\label{fig:3DRosenbluth}
	\end{figure}

	\begin{figure}[htp]
		\begin{center}
			\includegraphics[width=5.5in,]{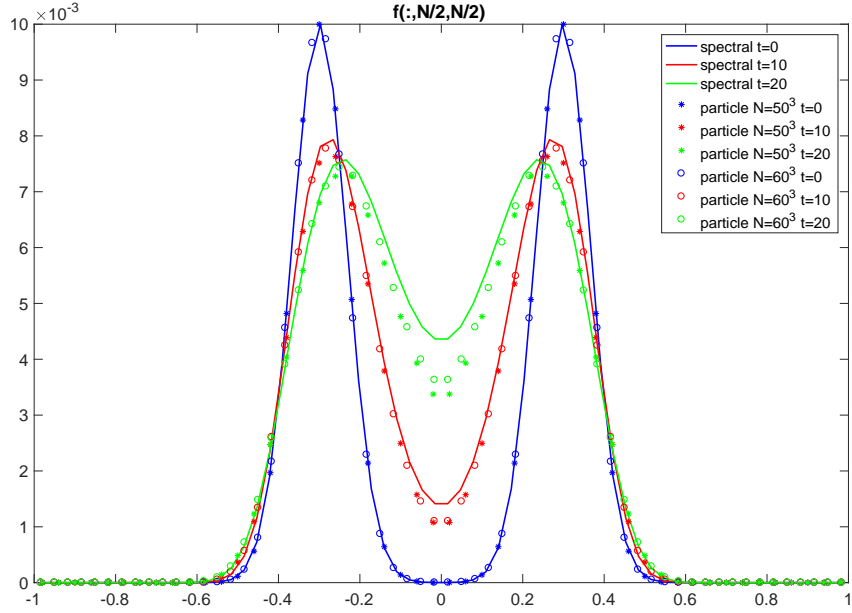}
		\end{center}
		\caption{Comparison of the particle method using treecode acceleration (using different particle numbers) with the spectral method ($N_v=64^3$). Slices of the solution at different times.}
		\label{fig:Rosenbluth}
	\end{figure}

	\appendix
	
	\section{BKW solutions for Maxwell molecules}\label{appa}
	
	We derive the BKW solution to the Landau equation \eqref{landau} in the Maxwell molecules case where $A(z) =B( |z|^{2} I - z \otimes z)$ whose kernel is spanned by $z$. Looking for solutions with the ansatz
	\begin{equation*}
	f(t,v)=\frac{1}{(2\pi K)^{d/2}}\exp \left ( -\frac{v^2}{2K}\right)(P+Qv^2), 
	\end{equation*}  
	where $K=K(t)$ is to be found, we require
	\begin{align*}
	&\rho=\int_{\mathbb{R}^d} f\,\rd{v}=P+dKQ=1,\\
	&T=\frac{1}{d}\int_{\mathbb{R}^d} fv^2\,\rd{v}=KP+(d+2)K^2Q=1\,.
	\end{align*}
	These conditions imply that
	\begin{equation*}
	P=\frac{(d+2)K-d}{2K} \quad \mbox{and} \quad Q=\frac{1-K}{2K^2}\,,
	\end{equation*}
	and therefore,
	\begin{equation*}
	f(t,v)=\frac{1}{(2\pi K)^{d/2}}\exp \left ( -\frac{v^2}{2K}\right)\left(\frac{(d+2)K-d}{2K}+\frac{1-K}{2K^2}v^2\right).
	\end{equation*}
	Direct differentiation yields
	\begin{equation} \label{ff}
	\frac{\partial f}{\partial t}=\frac{1}{(2\pi K)^{d/2}}\exp \left ( -\frac{v^2}{2K}\right) \left[d(d+2)K^2-2(d+2)Kv^2+v^4\right]\frac{1-K}{4K^4}K'.
	\end{equation}
	It is easy to check that
	\begin{align*}
	\nabla_v f&=\frac{1}{(2\pi K)^{d/2}}\exp \left ( -\frac{v^2}{2K}\right) \left(2Q-\frac{P+Qv^2}{K}\right)v,
	\end{align*}
	and hence
	\begin{align*}
	\nabla_v \log f=\frac{\nabla_v f}{f}=\frac{2Q}{P+Qv^2}v-\frac{1}{K}v.
	\end{align*}
	Therefore, we conclude that
	\begin{align*}
	\nabla_v \log f-\nabla_{v_*}\log f_*=2Q\frac{P(v-v_*)+Qv^2(v-v_*)+Q(v_*^2-v^2)v}{(P+Qv^2)(P+Qv_*^2)}-\frac{1}{K}(v-v_*).
	\end{align*}
	Using $A(z)z=0$, we have
	\begin{align*}
	A(v-v_*)\left[\nabla_v \log f-\nabla_{v_*}\log f_*\right]=2Q^2\frac{(v_*^2-v^2)A(v-v_*)v}{(P+Qv^2)(P+Qv_*^2)}.
	\end{align*}
	and then
	\begin{align*}
	A(v-v_*)\left[\nabla_v \log f-\right.&\left.\nabla_{v_*}\log f_*\right]ff_* =\frac{2Q^2}{(2\pi K)^{d}}\exp \left ( -\frac{v^2+v_*^2}{2K}\right)(v_*^2-v^2)A(v-v_*)v\nonumber\\
	&=\frac{2BQ^2}{(2\pi K)^{d}}\exp \left ( -\frac{v^2+v_*^2}{2K}\right)(v_*^2-v^2)\left[(v-v_*)^2I-(v-v_*)\otimes(v-v_*)\right]v.
	\end{align*}
	Hence we deduce that
	\begin{align*}
	\int_{\mathbb{R}^d} A(v-v_*) ff_* \left[ \nabla_v \log f- \nabla_{v_*} \log f_* \right]\, \rd{v_*}
	=\frac{2BQ^2}{(2\pi K)^{d/2}}\exp \left ( -\frac{v^2}{2K}\right)(I_1-v^2I_2),
	\end{align*}
	where
	\begin{align*}
	I_1&=\int_{\mathbb{R}^d}\frac{1}{(2\pi K)^{d/2}}\exp \left ( -\frac{v_*^2}{2K}\right)v_*^2\left[(v-v_*)^2v-(v-v_*)\otimes(v-v_*)v\right]\rd{v_*},\\
	I_2&=\int_{\mathbb{R}^d}\frac{1}{(2\pi K)^{d/2}}\exp \left ( -\frac{v_*^2}{2K}\right)\left[(v-v_*)^2v-(v-v_*)\otimes(v-v_*)v\right]\rd{v_*}.
	\end{align*}
	It can be checked that
	\begin{align*}
	I_1&=(dKv^2+(d^2+2d)K^2)v-(dKv^2+(d+2)K^2)v=(d+2)(d-1)K^2v,\\
	I_2&=(v^2+dK)v-(v^2+K)v=(d-1)Kv,
	\end{align*}
	and therefore, we conclude that
	\begin{align*}
	\int_{\mathbb{R}^d} A(v-v_*)\left[ \nabla_v \log f- \nabla_{v_*} \log f_* \right] ff_* \, \rd{v_*}
	=\frac{2BQ^2}{(2\pi K)^{d/2}}\exp \left ( -\frac{v^2}{2K}\right)(d-1)K[(d+2)K-v^2]v.
	\end{align*}
	Finally, we can write
	\begin{align} \label{QQm}
	\mathcal{Q}_L(f,f)(v)=&\nabla_{v}\cdot \int_{\mathbb{R}^d} A(v-v_*) \left[ \nabla_v \log f- \nabla_{v_*} \log f_* \right] ff_*\, \rd{v_*}\nonumber\\
	=& \frac{2BQ^2}{(2\pi K)^{d/2}}\exp \left ( -\frac{v^2}{2K}\right) (d-1)K\left[\frac{1}{K}v^4-2(d+2)v^2+d(d+2)K\right]\nonumber\\
	=&\frac{1}{(2\pi K)^{d/2}}\exp \left ( -\frac{v^2}{2K}\right)\frac{B(1-K)^2}{2K^4}(d-1)\left[v^4-2(d+2)Kv^2+d(d+2)K^2\right].
	\end{align} 
	Comparing (\ref{ff}) and (\ref{QQm}), we obtain $K'=2B(d-1)(1-K)$, which results in $K=1-C\exp(-2B(d-1)t)$. In 2D, we choose $C=1/2$ and $B=1/16$, then $K=1-\exp(-t/8)/2$. In 3D, we choose $C=1$ and $B=1/24$, then $K=1-\exp(-t/6)$.

	\section{Treecode for computing \eqref{vt3} in 3D}\label{appb}
	In view of \eqref{vt3}, the efficiency of the particle method is limited by five summations appearing on the right hand side. Indeed, if the total number of particles is $N$, then each summation needs $\mathcal O(N^2)$ computations, which is prohibitively expensive in 3D. To mitigate this issue, we developed a treecode method to accelerate the computation. The idea is that, first one partitions the particles into a hierarchy of clusters that has a tree structure (hence the number of cluster is $\mathcal O(\log N)$), then instead of using particle-particle interaction, one considers particle-cluster interaction, and therefore reduces the cost to $\mathcal O(N\log N)$ in total \cite{BH86}. 
	
	In general, consider the summation of the form 
	\begin{equation} \label{eqn:potential}
	U_i = \sum_{j = 1} ^N  q_j \phi(\vecx_i, \vecy_j), \quad i = 1 , 2, \cdots, N,
	\end{equation}
	where $\vecx_i$ and $\vecy_j$ are in $\RR^3$, and $\phi$ can be $\psi_\varepsilon$, three components of $\nabla \psi_\varepsilon$ or four components of $A$ in our case. Assume the particles have been divided into a hierarchy of clusters $C$, then the treecode evaluates the potential \eqref{eqn:potential} as a sum of particle-cluster interactions
	\begin{equation}\label{eqn:Vic}
	U_i = \sum_{c\in C} U_{i,c}, \qquad \mbox{where} \quad
	U_{i,c} = \sum_{w_j \in c} q_j \phi(\vecx_i, \vecy_j)\,.
	\end{equation}
	
	If particle $v_i$ and cluster $c$ are well separated (denoted below as MAC condition (\ref{eqn:MAC}) ), then the terms in (\ref{eqn:Vic}) can be expanded in Taylor series as:
	\begin{eqnarray} \label{eqn:Vic2}
	U_{i,c} &=& \sum_{w_j \in c} q_j \sum_{\| k \| =0}^\infty \frac{1}{k!} D_w^k \phi(\vecx_i, \vecy_c)(\vecy_j - \vecy_c)^k
	= \sum_{\| k \| =0}^\infty   \frac{1}{k!} D_w^k \phi(\vecx_i, \vecy_c) \sum_{\vecy_j \in c} q_j (\vecy_j - \vecy_c)^k  \nonumber
	\\ & \simeq &  \sum_{\| k \| =0}^p a^k(\vecx_i,\vecy_c) m_c^k,
	\end{eqnarray}
	where
	\begin{equation} \label{ak}
	a^k(\vecx_i,\vecy_c) = \frac{1}{k!} D_w^k \phi (\vecx_i, \vecy_c)
	\end{equation}
	is the $k$-th Taylor coefficient, and 
	\begin{equation*}
	m_c^k = \sum_{\vecy_j \in c} q_j (\vecy_j - \vecy_c)^k
	\end{equation*}
	is the $k$-th moment of cluster $c$ and the Taylor series has been truncated at order $p$. 
	
	Now let us specify the meaning of well-separated particle and cluster. Denote $R = |\vecx_i - \vecy_c|$, and $r_c = \max_{j\in c} |\vecy_c - \vecy_j|$, then in order for the Taylor series in (\ref{eqn:Vic2}) to converge, we need $r_c \leq R$. In theory, one can compute $U_{i,c}$ either by direct sum or via (\ref{eqn:Vic2}) depending on the accuracy and efficiency of (\ref{eqn:Vic2}). There is no standard way of choosing between the two as the optimal way is often problem dependent and one may find it by trial and error. However, a practical choice would be 
	\begin{equation}\label{eqn:MAC}
	\frac{r_c}{R} \leq \theta,
	\end{equation}
	where $\theta$ is a user-specified parameter for controlling the expansion error. This condition is called multipole acceptance criterion (MAC). 
	
	Coming back to our case \eqref{vt3}, it remains to compute $a^k$ from \eqref{ak} for different functions $\phi$. When $\phi$ is a Gaussian (i.e., $\psi_\varepsilon$) or gradient of a Gaussian (i.e., $\nabla \psi_\varepsilon$), one can derive the recursive relation for $a^k$ similarly as that in \cite{LJK09}, and here we only derive the ones for each component of matrix $A$. From here on, we denote $k = (k_1,k_2, k_3)$, and let $v = (v_1,v_2,v_3)$, $w=(w_1,w_2,w_3)$, then $A$ has the form 
	\[
	A = |v|^\gamma \left[\begin{array}{ccc} v_2^2+v_3^2  & -v_1v_2  & -v_1 v_3 \\ -v_1 v_2  & v_1^2+v_3^2 & -v_2v_3 \\ -v_1v_3 & -v_2v_3 & v_1^2 + v_2^2 \end{array} \right]\,.
	\]
	
	Let us first compute $a_{11}^k := \frac{1}{k!} D_w^k A_{11} (v-w)$. Taking the derivative of $\Aone$ in $w_1$, one has $D_{w_1} \Aone = \frac{\gamma (w_1- v_1)}{|w-u|^2} \Aone$, and hence 
	\begin{equation} \label{922}
	|w-v|^2 D_{w_1} \Aone = \gamma (w_1 - v_1) \Aone\,.
	\end{equation}
	Further differentiating the above equation $k_1-1$ times in $w_1$ , we obtain
	\[
	|w-v|^2 D_{w_1}^{k_1} \Aone + (w_1-v_1) [2(k_1-1)-\gamma] D_{w_1}^{k_1-1} \Aone + (k_1-1)(k_1-2-\gamma) D_{w_1}^{k_1-2} \Aone = 0\,.
	\]
	Taking the derivative $\Dtwo^{k_2} \Dthree^{k_3}$, the above relation becomes 
	\begin{align*}
	& |w-v|^2 D_w^k \Aone  + (w_1-v_1)[2(k_1-1)-\gamma] D_w^{k-e_1} \Aone + (k_1-1)(k_1-2-\gamma) D_w^{k-2e_1} A_{11} 
	\\ & \qquad +  2(w_2-v_2) k_2 D_w^{k-e_2} \Aone + k_2(k_2-1) D_w^{k-2e_2} \Aone 
	\\ & \qquad +  2(w_3-v_3) k_2 D_w^{k-e_3} \Aone + k_3(k_3-1) D_w^{k-2e_3} \Aone =0\,,
	\end{align*}
	where $e_1$, $e_2$ and $e_3$ are unit vectors in $\RR^3$. Dividing by $k!$ of the above equation, we have the following recursive relation for $a_{11}$:
	\begin{align*}
	& |w-v|^2 \aone^k  + 2(w_1 - v_1) \aone^{k-e_1} - \frac{2+\gamma}{k_1} (w_1-v_1) \aone^{k-e_1} + \left(1 - \frac{2+\gamma}{k_1} \right )\aone^{k-2e_1} \nonumber
	\\ & \qquad  + [2(w_2 - v_2) +1] \aone^{k-e_2} + [2(w_3-v_3)+1] \aone^{k-2e_3}=0, \quad k_1 \neq 0\,.
	\end{align*}
	Note the above relation is valid when $k_1 \neq 0$, and one needs to compute the case with $k_1=0$ separately. This can be done similarly by taking $\Dtwo^{k_2} \Dthree^{k_3}$ derivatives of \eqref{922} directly. 
	Likewise, we have for $a_{22}^k := \frac{1}{k!} D_w^k A_{22} (v-w)$:
	\begin{align*}
	& |w-v|^2 \atwo^k  + 2(w_2 - v_2) \atwo^{k-e_2} - \frac{2+\gamma}{k_2} (w_2-v_2) \aone^{k-e_2} + \left(1 - \frac{2+\gamma}{k_2} \right )\atwo^{k-2e_2} \nonumber
	\\ & \qquad  + [2(w_1 - v_1) +1] \atwo^{k-e_1} + [2(w_3-v_3)+1] \atwo^{k-2e_3}=0, \quad k_2 \neq 0\,.
	\end{align*}
	For $a_{33}^k := \frac{1}{k!} D_w^k A_{33} (v-w)$:
	\begin{align*}
	& |w-v|^2 \athree^k  + 2(w_3 - v_3) \athree^{k-e_3} - \frac{2+\gamma}{k_3} (w_3-v_3) \aone^{k-e_3} + \left(1 - \frac{2+\gamma}{k_3} \right )\athree^{k-2e_3} \nonumber
	\\ & \qquad  + [2(w_1 - v_1) +1] \athree^{k-e_1} + [2(w_2-v_2)+1] \athree^{k-2e_2}=0, \quad k_3 \neq 0\,.
	\end{align*} 
	For $a_{12}^k := \frac{1}{k!} D_w^k A_{12} (v-w)$:
	\begin{align*}
	& |w-v|^2 \aot^k  + 2(w_1 - v_1) \aot^{k-e_1} + \aot^{k-2e_1} + 2(w_2-v_2) \aot^{k-e_2} + \aot^{k-2e_2} \nonumber
	\\ & \qquad + \left(2- \frac{2+\gamma}{k_3}  \right)(w_3-v_3) \aot^{k-e_3} + \left( 1-\frac{2+\gamma}{k_3}\right) \aot^{k-2e_3} =0, \quad k_3 \neq 0\,.
	\end{align*} 
	For $a_{13}^k := \frac{1}{k!} D_w^k A_{13} (v-w)$:
	\begin{align*}
	& |w-v|^2 \aott^k  + 2(w_1 - v_1) \aott^{k-e_1} + \aott^{k-2e_1} + 2(w_3-v_3) \aott^{k-e_3} + \aott^{k-2e_3} \nonumber
	\\ & \qquad + \left(2- \frac{2+\gamma}{k_2}  \right)(w_2-v_2) \aott^{k-e_2} + \left( 1-\frac{2+\gamma}{k_2}\right) \aott^{k-2e_2} =0, \quad k_2 \neq 0\,.
	\end{align*} 
	For $a_{23}^k := \frac{1}{k!} D_w^k A_{23} (v-w)$:
	\begin{align*}
	& |w-v|^2 \attt^k  + 2(w_2 - v_2) \attt^{k-e_2} + \attt^{k-2e_2} + 2(w_3-v_3) \attt^{k-e_3} + \attt^{k-2e_3} \nonumber
	\\ & \qquad + \left(2- \frac{2+\gamma}{k_1}  \right)(w_1-v_1) \attt^{k-e_1} + \left( 1-\frac{2+\gamma}{k_1}\right) \attt^{k-2e_1} =0, \quad k_1 \neq 0\,.
	\end{align*}


	\section*{Acknowledgements}
	JAC, JH and LW would like to thank the American Institute of Mathematics for their support through a SQuaREs project where this work was finished. This research generated from the AIM workshop ``Nonlocal differential equations in collective behavior'' in June 2018.
	JH would like to thank Tong Ding for testing parameters in an undergraduate research project related to the current work. LW would like to thank Prof. Robert Krasny for fruitful discussion on treecode and Dr. Evan Bollig on the help with Minnesota super computers. 
	JAC was partially supported by EPSRC grant number EP/P031587/1 and the Advanced Grant Nonlocal-CPD (Nonlocal PDEs for Complex Particle Dynamics: 
	Phase Transitions, Patterns and Synchronization) of the European Research Council Executive Agency (ERC) under the European Union's Horizon 2020 research and innovation programme (grant agreement No. 883363). JH was partially supported by NSF DMS-1620250 and NSF CAREER grant DMS-1654152.
	LW was partially supported by NSF DMS-1903420 and NSF CAREER grant DMS-1846854. JW was funded by the President's PhD Scholarship program of Imperial College London.
	
	\bibliographystyle{abbrv}
	\bibliography{landau}

\end{document}